\documentclass[11pt]{article}
\usepackage{amssymb,amsmath,amsthm,amsfonts}
\usepackage{indentfirst}
\usepackage{amsmath}
\allowdisplaybreaks[4]
\usepackage{mathrsfs}
\usepackage{float}
\usepackage{amsfonts}
\usepackage{amssymb,amsmath,amsthm,amsfonts}
\usepackage{subfigure}
\usepackage{graphicx}
\usepackage{epstopdf}
\usepackage{color}
\usepackage{bm}
\usepackage{underscore}
\usepackage{ulem}

\usepackage{lineno}

\def\dref#1{(\ref{#1})}

\usepackage{epstopdf}

\newtheorem{theorem}{Theorem}[section]
\newtheorem{lemma}{Lemma}[section]

\numberwithin{equation}{section}

\newtheorem{remark}{Remark}[section]

\DeclareSymbolFont{txfontsA}{U}{txmia}{m}{it}
\SetSymbolFont{txfontsA}{bold}{U}{txmia}{bx}{it}
\DeclareFontSubstitution{U}{txmia}{m}{it}
\DeclareMathSymbol{\upalpha}{\mathord}{txfontsA}{"0B}
\DeclareMathSymbol{\upbeta}{\mathord}{txfontsA}{"0C}
\DeclareMathSymbol{\upgamma}{\mathord}{txfontsA}{"0D}
\DeclareMathSymbol{\updelta}{\mathord}{txfontsA}{"0E}
\DeclareMathSymbol{\upepsilon}{\mathord}{txfontsA}{"0F}
\DeclareMathSymbol{\upzeta}{\mathord}{txfontsA}{"10}
\DeclareMathSymbol{\upeta}{\mathord}{txfontsA}{"11}
\DeclareMathSymbol{\uptheta}{\mathord}{txfontsA}{"12}
\DeclareMathSymbol{\upiota}{\mathord}{txfontsA}{"13}
\DeclareMathSymbol{\upkappa}{\mathord}{txfontsA}{"14}
\DeclareMathSymbol{\uplambda}{\mathord}{txfontsA}{"15}
\DeclareMathSymbol{\upmu}{\mathord}{txfontsA}{"16}
\DeclareMathSymbol{\upnu}{\mathord}{txfontsA}{"17}
\DeclareMathSymbol{\upxi}{\mathord}{txfontsA}{"18}
\DeclareMathSymbol{\uppi}{\mathord}{txfontsA}{"19}
\DeclareMathSymbol{\uprho}{\mathord}{txfontsA}{"1A}
\DeclareMathSymbol{\upsigma}{\mathord}{txfontsA}{"1B}
\DeclareMathSymbol{\uptau}{\mathord}{txfontsA}{"1C}
\DeclareMathSymbol{\upupsilon}{\mathord}{txfontsA}{"1D}
\DeclareMathSymbol{\upphi}{\mathord}{txfontsA}{"1E}
\DeclareMathSymbol{\upchi}{\mathord}{txfontsA}{"1F}
\DeclareMathSymbol{\uppsi}{\mathord}{txfontsA}{"20}
\DeclareMathSymbol{\upomega}{\mathord}{txfontsA}{"21}
\DeclareMathSymbol{\upvarepsilon}{\mathord}{txfontsA}{"22}
\DeclareMathSymbol{\upvartheta}{\mathord}{txfontsA}{"23}
\DeclareMathSymbol{\upvarpi}{\mathord}{txfontsA}{"24}
\DeclareMathSymbol{\upvarrho}{\mathord}{txfontsA}{"25}
\DeclareMathSymbol{\upvarsigma}{\mathord}{txfontsA}{"26}
\DeclareMathSymbol{\upvarphi}{\mathord}{txfontsA}{"27}

\DeclareSymbolFont{ugmL}{OMX}{mdugm}{m}{n}
\SetSymbolFont{ugmL}{bold}{OMX}{mdugm}{b}{n}
\DeclareMathAccent{\wideparen}{\mathord}{ugmL}{"F3}

\def\dref#1{(\ref{#1})}

\def\pt{\partial}

\def\ra{\rightarrow}

\def\s{\subseteq}

\def\e{\varepsilon }
\def\mb{\mathbb}
\def\ol{\overline}

\def\vp{\varphi}

\def\fa{\forall}

\def\bf{\textbf}
\def\pt{\partial}

\def\om{\omega}
\def\Om{\Omega}

\def\al{\alpha}
\def\be{\beta}
\def\de{\delta}
\def\ga{\gamma}

\def\Ga{\Gamma}
\def\La{\Lambda}

\def\ts{\times}

\def\iy{\infty}

\def\f{\frac}

\def\se{\setminus}

\def\df{\mathrm d}

\def\wh{\widehat}

\def\mcH{\mathcal{H}}

\def\mcO{\mathcal{O}}

	\DeclareMathOperator{\Div}{div}

	\DeclareMathOperator{\dist}{dist}

	\DeclareMathOperator{\supp}{{supp}}

	\newcommand{\R}{\mathbb R}

\begin{document}

		\title{\bf  {\bf On Controllability of  a Class of $N$-dimensional Hyperbolic Equations
				with Internal Single-point Degeneracy} }

		\author{Donghui Yang$^{a}$,  Weijia Wu$^{b}$\footnote{\small
				The corresponding author. Email: weijiawu@yeah.net},
			\\ $^a${\it School of Mathematics and Statistics, Central South University}\\
			{\it Changsha 410075, P.R.China} \\
			$^b${\it School of Mathematics and Physics}\\
			{\it North China Electric Power University, Beijing 102206, China}\\
			 		}
		\date{}
		
		\maketitle

		\begin{abstract}
			   This paper explores the controllability of a class of $N$-dimensional hyperbolic equations featuring a single interior degenerate point. Firstly, we establish the well-posedness of the equation through the application of the Hardy inequality. Following this, we primarily utilize the Carleman estimate method to derive the observability inequality. By leveraging the equivalence between observability and controllability, we deduce the exact controllability of the equation. It is worth noting that our selected control region includes the degenerate null point. In the Carleman estimate, we adopt a unique approach to construct the weight function, effectively negating the influence of the degenerate region.
			\vspace{0.2cm}
			
			  \vspace{0.2cm}
			
			
		\end{abstract}
		\thispagestyle{empty}		
		\section{Introduction}
		
Controllability, a basic   concept in dynamical systems, holds utmost importance in understanding and influencing the trajectories of systems under specified control inputs. In recent times, the study of controllability within the framework of degenerate equations has sparked extensive and profound mathematical inquiries. This heightened scrutiny is attributed to the unique capabilities of these equations in depicting specific behaviors pertinent to natural and physical phenomena. For instance, the renowned Crocco equation in meteorology exemplifies a degenerate parabolic equation (\cite{martinez2003regional}), which finds extensive application in fields such as hypersonic aerodynamics and fluid dynamics. Similarly, the Black-Scholes equation (\cite{sakthivel2008exact}), ubiquitous in finance, presents a degenerate form, further emphasizing the versatility and applicability of degenerate equations across diverse disciplines.

Unlike classical  equations, degenerate equations violate the uniform ellipticity condition, thereby complicating the establishment of solution existence, uniqueness, and regularity. The degeneration of coefficients may result in significant singularities in certain regions, posing additional hurdles in proving controllability.

		  Compared to research on the controllability of degenerate parabolic equations (\cite{ banerjee2022carleman, cannarsa2013null, beauchard20152d,cannarsa2016global, flores2010carleman}), there has been comparatively less exploration into the controllability of degenerate hyperbolic equations, with current focus primarily on one-dimensional cases. In \cite{gueye2014exact}, the author employed nonharmonic Fourier series to investigate sharp observability estimates for a class of one-dimensional degenerate wave equations, with control applied to degenerate boundaries. In contrast, \cite{zhang2017persistent} used the multiplier method to study the persistent regional null controllability of a class of degenerate wave equations, where control acts internally.  In \cite{gueye2014exact}, control is observed to act on the degenerate boundary, whereas in \cite{zhang2017persistent}, control acts internally but still involves boundary degeneracy. Differing from both \cite{gueye2014exact} and \cite{zhang2017persistent}, the work presented in \cite{bai2023exact} explores the exact controllability of a one-dimensional degenerate wave equation, with control applied to the boundary but the degenerate point located within the interior of the domain. In \cite{bai2023exact}, both strong and weak degeneracy cases are considered, and exact controllability is proven under weak degeneracy. However, when control is applied only at one boundary point, exact controllability under strong degeneracy cannot be achieved. Unlike \cite{bai2023exact}, \cite{Allal2022} considers not only the degenerate case but also the singular case. Despite significant progress in studying one-dimensional degenerate hyperbolic equations (\cite{alabau2017control, bai2020exact, gueye2014exact, zhang2017persistent, zhang2018interior}), the application of Carleman estimates in this context remains relatively limited.

 The primary approach in Carleman estimates involves constructing a weight function to derive an estimate for the solution. In the context of degeneration, constructing this weight function becomes more complex. In this paper, we introduce a novel weight function that distinguishes itself from those utilized for equations satisfying the uniform ellipticity condition (as seen in \cite{Baudouin2007,FU2007}) and also from those employed in other studies focusing on degenerate Carleman estimates (such as \cite{alabau2006control,alabau2017control,Fragnelli2021}). This innovative approach enables us to effectively mitigate the influence of degenerate points and attain the desired weighted Carleman estimate.

  In this paper, we study the controllability problem for a class of high-dimensional internal single-point degenerate hyperbolic equations. Our primary strategy involves leveraging the Carleman estimate method, coupled with the introduction of suitable weight functions, to achieve
  exact  controllability for this specific class of problems. Our work stands out from existing research, such as \cite{bai2023exact,gueye2014exact,zhang2017persistent}, as it focuses on the high-dimensional case and presents distinctive features in both the control domain and the methodology employed. Furthermore, to our knowledge, there has been no prior research on the controllability of high-dimensional degenerate hyperbolic equations, making this paper a first attempt  in this topic.

		The remaining structure of this paper is outlined as follows. Section \ref{Se2} expounds on the principal findings of this study. In Section \ref{Se3}, several well-posedness conclusions are presented. Section \ref{Se4} is devoted to the construction of a corresponding Carleman estimate, upon which, in the subsequent Section \ref{Se5}, we substantiate unique continuation and approximate controllability. Furthermore, Section \ref{Se6} provides conclusions on observability inequalities and exact controllability, followed up by a concluding remarks in Section \ref{Se7}.

		\section{Problems and main results}\label{Se2}
		
		In this paper, we examine the following single-point interior degenerate system:
		\begin{equation}\label{EQ-1}
			\begin{cases}
				\pt_t^2 u-\Div\left(|x|^\alpha \nabla u\right)=\chi_\omega f, & \mbox{in } Q, \\
				u=0, & \mbox{on } \Sigma,\\
				u(0)=u_0, \pt_tu(0)=u_1, &\mbox{in }\Om,
			\end{cases}
		\end{equation}
		where  $\Om\s\R^N$  is a smooth domain with $N\geq 1$ being an integer and $0\in \Om$. Here, $Q=\Om\ts (0,T)$, with $T$ being a specific time that will be detailed later in Chapter \ref{Se5}. Additionally, $\Sigma=\pt Q=\pt\Om\ts (0,T)$, $\om\s\Om$ is a given control domain that will be specified later, $u_0\in \mcH_0^1(\Om)$ (with  $\mcH_0^1(\Om)$ defined subsequently), and
		$u_1\in L^2(\Om)$  represent the initial data. The control function $f$ belongs to $f\in L^2(Q)$. Let $\nu$ denote the outer normal vector to  $\pt\Om$, and define
		\begin{equation*}
			\Ga_+:=\left\{x\in\pt\Om\bigm| x\cdot \nu(x)\geq 0\right\},
		\end{equation*}
		where $\Ga_+$ is an open subset of $\Ga=\pt\Om$. Set
		\begin{equation*}
			\mcO(\Ga_+,\de)=\left\{x\in\Om\bigm| \dist(x,\Ga_+)<\de \right\},
		\end{equation*}
		and assume that
		\begin{equation}\label{2-6}
			\mcO(\Ga_+, 3\de)\s \om,
		\end{equation}
		where $\de>0$ is a given constant,  and
		\begin{equation*}
			\dist(x, \Ga_+)=\inf_{y\in \Ga_+}|x-y|.
		\end{equation*}
		Define
		\begin{equation*}
			\mcH^1(\Om)=\left\{u\in L^2(\Om)\left| \ \! |x|^\al \nabla u\cdot \nabla u\in L^2(\Om)\right.\right\},
		\end{equation*}
		with the norm
		\begin{equation*}
			\|u\|_{\mcH^1(\Om)}=\left\| |x|^{\frac{\alpha}{2}} \nabla u \right\|_{L^2(\Omega)}+\|u\|_{L^2(\Omega)}.
		\end{equation*}
		and
		\begin{equation*}
			\mcH_0^1(\Om)=\overline{C_0^\infty(\Omega)}^{\left\|\cdot \right\|_{\mcH^1(\Om)} }.
		\end{equation*}
		It is evident that there exists a solution
		\begin{equation*}
			u\in C([0,T]; \mcH_0^1(\Om))\cap C^1([0,T]; L^2(\Om))
		\end{equation*}
		to the  equation \eqref{EQ-1} for any  $(u_0,u_1)\in \mcH_0^1(\Om)\times  L^2(\Om)$.  This will be proven using the classic Galerkin method in Section \ref{Se3}.
		
		As is classical in controllability problems (HUM), we introduce the nonhomogeneous adjoint problem for \eqref{EQ-1}:
		\begin{equation}\label{EQ-2}
			\begin{cases}
				\pt_t^2z-\Div(|x|^\al \nabla z)=F, &\mbox{in } Q,\\
				z=0, &\mbox{on }\Sigma,\\
				z(0)=z_0, \pt_tz(0)=z_1, &\mbox{in }\Om,
			\end{cases}
		\end{equation}
		where $F\in L^2(\Om)$, $z_0\in \mcH_0^1(\Om)$,  and $z_1\in L^2(\Om)$.
		
		The main results of this paper are the Carleman inequality and observability inequality stated in Theorem \ref{Carleman2}
		following.			
		\begin{theorem}\label{Carleman2}
			Let $\om\subset\Om$ be given as in \eqref{2-6}. Assume  $0\in\om$. For any solution $z$ to \eqref{EQ-2}, there exist positive constants  $C=C(\om,\Om)$, such that for sufficiently large constants
			$s$ and $\gamma$, the following inequality holds:
			\begin{equation}\label{2.3}
				\begin{split}
					&\iint_Qe^{2s\vp}s\ga \vp \left(|x|^\al \nabla z\cdot \nabla z+|z_t|^2+s^2\ga^2\vp^2|z|^2\right)\df x\df t\\
					&\leq C\iint_{\om\ts (0,T)}e^{2s\vp}s\ga \vp\left( z_t^2+|x|^\al \nabla z\cdot \nabla z +s^2\ga^2\vp^2|z|^2 \right)\df x\df t+C\iint_Qe^{2s\vp} F^2\df x\df t,
				\end{split}
			\end{equation}
	where     $\varphi$  is the weight function defined in \eqref{vp}.
		\end{theorem}
		
		Using the duality between controllability and observability   (\cite{rockafellar1967duality,lions1992remarks}), proving controllability is equivalent to establishing an observability property for the adjoint system. From this, we can deduce the unique continuation and exact controllability of \eqref{EQ-1}.

		\begin{theorem}\label{Carleman5}
			Let  $\om$  be given as in \eqref{2-6} and $0\in \om$. If $z$ is a solution of \eqref{EQ-2} with $F=0$ (or, equivalently, if $u$ is a solution of \eqref{EQ-1} with $f=0$), and $z=0$ on $\omega \times \left( 0,T\right) $, then $z=0$ in $Q$.
		\end{theorem}
		After establishing unique continuation, we can further deduce controllability.
		\begin{theorem}\label{Carleman4}
			Let $\om\s\Om$ be given as in \eqref{2-6}.
			Assume $0\in \om$. Then the system \eqref{EQ-1} is exactly controllable, i.e., for every $(u_0,u_1), (w_0,w_1) \in \mcH_0^1(\Om) \times L^2(\Omega)$,  there exists a control $f\in L^2(Q)$,  such that the solution of \eqref{EQ-1} satisfies $u(T)=w_0$ and $ u_t (T)=w_1$.				
		\end{theorem}

		\section{Well-posedness}\label{Se3}
		
		In this section, we discuss  the well-posedness of \eqref{EQ-1}. To set the stage, we present some preliminary results.
		
 Consider  $\varepsilon _0\in (0, d(0,\Gamma))$,  where
 $d(0,\Gamma)=\inf_{y\in \Gamma} |y|$,  and for a small constant $\varepsilon \in \left( 0,\varepsilon _0\right] $, let
		\begin{equation*}
			\Om^\varepsilon  = \left\lbrace x\in \Om \mid |x|>\varepsilon  \right\rbrace, \ \Om_\varepsilon  = \left\lbrace x\in \Om \mid |x|<\varepsilon  \right\rbrace \mbox{ and } S(\varepsilon ) = \left\lbrace x\in \mathbb{R}^N \bigm| |x|=\varepsilon  \right\rbrace\,.
		\end{equation*}
		\begin{lemma}\label{L2.2}
			For any $N\ge 2$ and $\alpha\in (0,2)$, it holds that  $|x|^{\frac{\alpha}{2}-1}  u  \in L^2(\Om)$ for all $ u \in \mcH_0^1(\Om)$, satisfying the inequality
			\begin{equation}\label{2.1}
				(N-2+\alpha) \left\| |x|^{\frac{\alpha}{2}-1}   u  \right\| _{L^2(\Om)} \le C\left\|  u  \right\| _{\mcH_0^1(\Om)},
			\end{equation}
			where $C>0$ is a constant independent of   $ u$.  Additionally, if $ u \in \mathcal{H}_0^1(\Om)$,
			then $ u \in L^2(\Om)$ .
		\end{lemma}
		\begin{proof} 		
			For $ u  \in \mcH_0^1(\Om)$,  its restriction belongs to $W^{1,2}\left(\Om^{\varepsilon }\right)$,
			and its trace represents a bounded linear map into  $L^2\left(\partial \Om^{\varepsilon }\right)$. Then,
since the trace of $ u $ is zero on $\partial \Omega$, we have
			$$
			\begin{aligned}
				2 \int_{\Om^{\varepsilon }} |x|^{\alpha-2}   u (x \cdot \nabla  u ) d x & =\int_{\Om^{\varepsilon }} |x|^{\alpha-2}  x \cdot \nabla\left( u ^2\right) d x \\
				& =-\int_{S(\varepsilon )} |x|^{\alpha-1}   u ^2 d s-\int_{\Om^{\varepsilon }}(N-2+\alpha) |x|^{\alpha-2}   u ^2 d x,
			\end{aligned}
			$$
			 and
			$$
			\begin{aligned}
				(N-2+\alpha) \int_{\Om^{\varepsilon }} |x|^{\alpha-2}   u ^2 d x & \le-2 \int_{\Om^{\varepsilon }} |x|^{\alpha-2}   u (x \cdot \nabla  u ) d x \le 2 \int_{\Om^{\varepsilon }}\left(|x|^{\frac{\alpha}{2}-1}  | u |\right)\left(|x|^{\frac{\alpha}{2}}|\nabla  u |\right) d x \\
				& \le 2\left\{\int_{\Om^{\varepsilon }} |x|^{\alpha-2}  u ^2 d x\right\}^{1 / 2}\left\{\int_{\Om^{\varepsilon }} |x|^\alpha   \nabla  u  \cdot \nabla  u   d x\right\}^{1 / 2}.
			\end{aligned}
			$$
			Finally, \eqref{2.1} follows by letting $\varepsilon  \rightarrow 0^+$ due to  $|x|^\alpha \nabla  u \cdot\nabla  u  \in L^1(\Omega)$.
		\end{proof}		
		
		The inequality \dref{2.1} plays a pivotal role and will be essential in the subsequent proof of compact embedding. Now, let us establish that the solution space is indeed a Hilbert space.

		\begin{lemma}\label{HS}	
			The space $(\mcH_0^1(\Omega),\left\langle \cdot,\cdot\right\rangle _{\mcH_0^1(\Omega)})$ is a Hilbert space.
		\end{lemma}
		
		\begin{proof}	
			Firstly, it is straightforward to verify that  $(\mcH_0^1(\Omega),\left\langle \cdot,\cdot\right\rangle_{\mcH_0^1(\Omega)})$ is an inner product space. 			Next, we prove that
			$(\mcH_0^1(\Omega),\left\langle \cdot,\cdot\right\rangle_{\mcH_0^1(\Omega)})$  is a Hilbert space. For simplicity, we define the weighted space
			\begin{equation*}
				L_\alpha^2(\Om):=\left\lbrace w\left|\ \! w: \Om\ra \mb{R}^N \mbox{ is a measurable function}, \mbox{ and } \int_{\Om} |x|^\alpha   w\cdot w dx  < \infty\right. \right\rbrace.
			\end{equation*}
			It is easily checked that $L_\alpha^2(\Om)$ is a Hilbert space with the inner product
			\begin{equation*}
				(w_1,w_2)=\int_\Om |x|^\alpha  w_1\cdot w_2dx=\int_\Om \left(|x|^\frac{\alpha}{2} w_1\right)\cdot\left(|x|^\frac{\alpha}{2}w_2\right)dx.
			\end{equation*}
			
			Let $\{v_n\}_{n\in\mb{N}}\subset\mcH_0^1(\Omega)$  be a Cauchy sequence. Then there exist $v\in L^2(\Omega)$ and  $g=(g_1,g_2,\dots,g_N)\in L^2(\Omega)^N$ such that
			\begin{equation*}
				v_n\ra v \mbox{ strongly in } L^2(\Om), \mbox{and } |x|^{\frac{\alpha}{2}}\nabla v_n \ra g \mbox{ strongly in } L^2(\Omega)^N.
			\end{equation*}
			To conclude, we need to show that $\nabla v = |x|^{-\frac{\alpha}{2}}g$.  For this purpose, since
			$v_n\ra v  $  strongly in $L^2(\Omega)$,  we have $\frac{\partial v_n}{\partial x_i}\rightarrow \frac{\partial v}{\partial x_i}$
			in the sense of distributions for $i=1,2,\cdots, N$, and the distributional limit is unique. Therefore, it suffices to prove that
			\begin{equation*}
				\nabla v_n\ra |x|^{-\frac{\alpha}{2}}g  \mbox{ strongly in } L_\alpha^2(\Om) \Rightarrow \nabla v_n\ra  |x|^{-\frac{\alpha}{2}}g \mbox{ strongly in } (C_0^\infty(\Omega))'.
			\end{equation*}
			Indeed,
			for all $\varphi\in C_0^\infty(\Omega)$,  we have
			\begin{eqnarray*}
				&&\left|\int_\Om (\nabla v_n  - |x|^{-\frac{\alpha}{2}}g)\varphi dx \right|\\
				&&\le \left\| \varphi\right\|_{L^\infty (\Om)}  \int_{\Om} \left| \nabla v_n  - |x|^{-\frac{\alpha}{2}}g \right| dx\le\|\varphi\|_{L^\infty(\Om)}\int_\Om \left||x|^{-\frac{\alpha}{2}}\right| \left||x|^\frac{\alpha}{2}\nabla v_n-g\right|dx\\
				&&\le C\|\varphi\|_{L^\infty(\Om)}\left(\int_\Om |x|^{-\alpha}d x\right)^\frac{1}{2}\left(\int_\Om \left||x|^\frac{\alpha}{2}\nabla v_n-g\right|^2dx\right)^\frac{1}{2}\\
				&&\le C\|\varphi\|_{L^\infty(\Om)}\left\|\nabla v_n-|x|^{-\frac{\alpha}{2}}g\right\|_{L_\alpha^2(\Om)}\rightarrow 0, \ n\rightarrow\infty,
			\end{eqnarray*}
			where $C$ denotes different constants depending on the context.
		\end{proof}

		Below, we will prove a significant compact embedding theorem, which is vital for our application of the Galerkin method to demonstrate the existence of solutions.
		
		\begin{theorem}\label{compactembed}
			$\mcH_0^1(\Omega)$ is compactly embeded in $L^2(\Omega)$.
		\end{theorem}
		\begin{proof}
			To establish the compactness of the embedding it is suffices to show that if $\left\{ u_n\right\}$ is a sequence converging weakly to  zero in $\mcH_0^1(\Om)$ as $n \rightarrow \infty$, then $\left\| u_n\right\|_{L^2(\Om)} \rightarrow 0$ as $n \rightarrow \infty$  by abstract subsequence.
			
			Since $\mcH_0^1(\Om)$ is continuously embedded in $L^2(\Omega)$, $L^2(\Omega)^* \subset$ $\mcH_0^1(\Om)^*$ and hence $\left\{ u_n\right\}$ converges weakly to zero in $L^2(\Omega)$.

			Consider $\varepsilon  \in\left(0, \varepsilon _0\right]$.  If the sequence
			$\left\{ u_n\right\}$  does not converge weakly to zero in  $W^{1,2}\left(\Om^{\varepsilon }\right)$,
			then there exist  $f \in W^{1,2}\left(\Om^{\varepsilon }\right)^*$, a subsequence
			$\left\{ u_{n_k}\right\}$ and $\delta>0$  such that
			$\left|f\left( u_{n_k}\right)\right| \geq \delta$
			for all $n_k$. By passing to a further subsequence if necessary, we can assume that
			$\left\{ u_{n_k}\right\}$  converges weakly to an element $v$ in  $W^{1,2}\left(\Om^{\varepsilon }\right)$. Consequently,
			$\left\{ u_{n_k}\right\}$  also converges weakly to $v$ in  $L^2\left(\Om^{\varepsilon }\right)$.  Since $\left\{ u_n\right\}$
			converges weakly to zero in $L^2(\Omega)$ and hence also in  $L^2\left(\Om^{\varepsilon }\right)$,
			it follows that $v=0$ almost everywhere on  $\Om^{\varepsilon }$. However, this implies that
			$f\left( u_{n_k}\right) \rightarrow f(v)=f(0)=0$ as $n_k \rightarrow \infty$, which contradicts the choice of $\delta$. Therefore,
			$\left\{ u_n\right\}$  must converge weakly to zero in  $W^{1,2}\left(\Om^{\varepsilon }\right)$ and, consequently,
			$\left\| u_n\right\|_{L^2\left(\Om^{\varepsilon }\right)} \rightarrow 0$ as $n \rightarrow \infty$.  Therefore,
			\begin{equation}\label{UN}
				\limsup _{n \rightarrow \infty}\left\| u_n\right\|_{L^2(\Om)}^2=\limsup _{n \rightarrow \infty} \int_{B(0, \varepsilon )}\left| u_n\right|^2 d x.
			\end{equation}
			However, from \cite{catrina2001caffarelli} and Lemma \ref{L2.2}, we have
			\begin{equation}\label{LP}
				\left\| u \right\|_{L^q(\Om)} \le C \left\|  u\right\|_{\mcH_0^1(\Om)}, \ 1\le q \le \frac{2N}{N-2+\alpha},
			\end{equation}
			and by taking $q>2$, we obtain
			\begin{equation*}
				\int_{B(0, \varepsilon )}\left| u_n\right|^2 d x \le \left(\int_{B(0, \varepsilon )} |1|^ {\frac{q}{q-2}} dx\right) ^{\frac{q-2}{q}} \left( \int_{B(0, \varepsilon )} \left( | u_n|^2 \right) ^{\frac{q}{2}}dx\right)^{\frac{2}{q}}
				\le |B(0, \varepsilon )|^ {\frac{q-2}{q}} \left\|  u_n \right\| ^2_{L^q(\Om)}.
			\end{equation*}
			The weak convergence of  $\left\{ u_n\right\}$ in $\mcH_0^1(\Om)$ and \eqref{LP} imply that this sequence is bounded in
			$L^q(\Omega)$.  Since $q>2$, we have
			$$
			\int_{B(0, \varepsilon )}\left| u_n\right|^2 d x \le C |B(0,\varepsilon )|^\frac{q-2}{q}\left\|  u_n\right\|^{2}_{\mcH_0^1(B(0, \varepsilon ))}\leq C |B(0,\varepsilon )|^\frac{q-2}{q}.
			$$
			Letting $\varepsilon  \rightarrow 0^+$ in \eqref{UN} shows that $\left\| u_n\right\|_{L^2(\Om)} \rightarrow 0$ as $n \rightarrow \infty$, completing the proof.
		\end{proof}

		Since  we have established in Lemma \ref{compactembed} that the solution space
		$\mcH_0^1(\Omega)$  is compactly embedded in $L^2(\Omega)$.
		we can readily demonstrate the existence of a solution for \eqref{EQ-1}. To do so, we first present the following lemma.

		\begin{lemma}\label{L2.3}
			Let $\Om\s \R^N$ be a domain with smooth boundary. Then, there exists a $C^2$-vector field $V$ such that
			\begin{equation*}
				V(x)=\nu(x), \ x\in \pt\Om, \ \mbox{ and } |V(x)|\leq 1, \ x\in \ol{\Om},
			\end{equation*}
			where $\nu$   is the unit outward normal vector to $\pt\Om$.
		\end{lemma}
		\begin{proof}
			Refer to \cite[Lemma 2.3, p. 61]{Bellassoued2017}.
		\end{proof}

		Next, we will focus on proving the existence of solutions to Problem \eqref{EQ-1}. Hereafter, we present the main results of this section.
		
		\begin{theorem}\label{T2.2}
			Given $T>0$, suppose that
			\begin{equation*}
				\chi_\omega f\in L^2(Q), \quad u_0\in \mcH_0^1(\Om), \ \mbox{ and } u_1\in L^2(\Om).
			\end{equation*}
			Then, there exists a unique solution $u$ to  \eqref{EQ-1} such that
			\begin{equation*}
				u\in C([0,T]; \mcH_0^1(\Om))\cap C^1([0,T]; L^2(\Om)),
			\end{equation*}
			and there exists a constant $C>0$ satisfying
			\begin{equation}\label{10.02.1}
				\|u\|_{C([0,T]; \mcH_0^1(\Om))}+\|\pt_tu\|_{C([0,T]; L^2(\Om))}\leq C\left(\|u_0\|_{\mcH_0^1(\Om)}+\|u_1\|_{L^2(\Om)}+\|\chi_\omega f\|_{L^2(Q)}\right).
			\end{equation}
			Furthermore,
			\begin{equation*}
				\pt_\nu u\in L^2(\Sigma),
			\end{equation*}
			and there exists a constant $C=C(T, \Om)>0$ such that
			\begin{equation}\label{10.02.4}
				\|\pt_\nu u\|_{L^2(\Sigma)}\leq C\left(\|u_0\|_{\mcH_0^1(\Om)}+\|u_1\|_{L^2(\Om)}+\|\chi_\omega f\|_{L^2(Q)}\right).
			\end{equation}
		\end{theorem}

		\begin{proof}   The proof will be divided into two major  steps.
			
			{\bf {\bf Step 1:}}  Utilizing the well-known result on the unique existence of weak solutions for abstract evolution equations presented in \cite{evans2010} or the Galerkin method, we derive
			\begin{equation*}
				u\in C([0,T]; \mcH_0^1(\Om))\cap C^1([0,T]; L^2(\Om)).
			\end{equation*}
			Multiplying the first equation of \eqref{EQ-1} by  $\pt_t u$  and applying Green's formula, we obtain
			\begin{equation*}
				\f{\df}{\df t}\int_\Om \Big(|u_t(t)|^2+|x|^\al \nabla u(t)\cdot \nabla u(t)\Big)\df x=2\int_\Om \chi_\omega f(x,t) u_t(t)\df x.
			\end{equation*}
			Define
			\begin{equation*}
				E(t)=\left(\left\||x|^\al \nabla u(t)\cdot \nabla u(t)\right\|_{L^1(0,T; L^2(\Om))}+\|u_t(t)\|_{L^2(\Om)}^2\right)^\f{1}{2}
			\end{equation*}
			for $t\in [0,T]$. Then,
			\begin{equation*}
				\f{\df}{\df t}E^2(t)\leq C\left(\|\chi_\omega f(\cdot, t)\|_{L^2(\Om)} E(t)+E^2(t)\right), \ t\in (0,T),
			\end{equation*}
			which simplifies to
			\begin{equation*}
				E'(t)\leq C\left(\|\chi_\omega f(\cdot, t)\|_{L^2(\Om)}+E(t)\right), \ t\in (0,T).
			\end{equation*}
			By applying Gronwall's inequality and integrating over $t\in [0,T]$, we get
			\begin{equation}\label{10.02.3}
				E(t)\leq C_T\left(E(0)+\|\chi_\omega f\|_{L^2(0,T; L^2(\Om))}\right),\ t\in (0,T).
			\end{equation}
			This implies \eqref{10.02.1} via the   Poincar\'e inequality.
			
			{\bf {\bf Step 2:}}  We define a mapping
			\begin{equation*}
				\La: H_0^1(\Om)\ts L^2(\Om)\ts L^2(Q)\ra L^2(\Sigma), \  (u_0, u_1, \chi_\omega f)\mapsto \pt_\nu u,
			\end{equation*}
			where $u$ is the solution of \eqref{EQ-1} and $H_0^1(\Om)$ denotes the standard Sobolev space. We aim to prove that $\chi_\omega f$  is a bounded linear operator, i.e.,
			\begin{equation}\label{10.02.2}
				\|\pt_\nu u\|_{L^2(\Sigma)}\leq C\left(\|u_0\|_{\mcH_0^1(\Om)}+\|u_1\|_{L^2(\Om)}+\|\chi_\omega f\|_{L^2(Q)}\right).
			\end{equation}
			Given the density of $C_0^\iy(\Om)$   in $\mcH_0^1(\Om)$
			and a density argument, it suffices to demonstrate the result for  $u_0\in H^2(\Om)\cap \mathcal{H}_0^1(\Om), u_1\in \mcH_0^1(\Om)$,
			and $\chi_\omega f\in L^2(0,T; \mcH_0^1(\Om))$. Select a function $\zeta\in C^3(\ol\Om)$ that satisfies
			\begin{equation*}
				0\leq \zeta\leq 1 \mbox{ on }\R^N, \quad \zeta\equiv 0 \mbox{ on } B(0,\e), \quad \zeta\equiv 1 \mbox{ on }\R^N-B(0,2\e), \ \mbox{ and } B(0,3\e)\s \Om.
			\end{equation*}
			Replace $V$ by $\zeta V$ in Lemma \ref{L2.3} (still denoting  $\zeta V$ by $V$), yielding a
			$C^2$-vector field  $V$ on $\ol\Om$  that satisfies
			\begin{equation*}
				V(x)=\nu(x), \ x\in\pt\Om; \quad |V(x)|\leq 1,  \ x\in \Om; \quad \supp V\s \Om\se B(0,\e).
			\end{equation*}
			Multiplying the first equation of \eqref{EQ-1} by $V\cdot\nabla u$ and integrating over $Q$, we obtain
			\begin{eqnarray*}
				H
				&=&\iint_Q\chi_\omega f(x,t)(V\cdot\nabla u)\df x\df t\\
				&=&\iint_Q u_{tt}(V\cdot\nabla u)\df x\df t-\iint_Q (V\cdot\nabla u)\Div\left(|x|^\al \nabla u\right)\df x  \df t=: H_1+H_2.
			\end{eqnarray*}
			Calculating $H_1$, we get
			\begin{eqnarray*}
				H_1
				&=&\iint_Q u_{tt}(V\cdot\nabla u)\df x\df t=\iint_Q\pt_t\left(u_t(V\cdot\nabla u)\right)\df x\df t-\iint_Q u_t V\cdot \nabla u_t\df x\df t\\
				&=&\int_\Om u_t(V\cdot\nabla u)\df x\Big|_0^T-\f{1}{2}\iint_Q V\cdot \nabla u_t^2\df x\df t\\
				&=&\int_\Om u_t(V\cdot\nabla u)\df x\Big|_0^T-\f{1}{2}\iint_Q\Div\left(V u_t^2\right)\df x\df t+\f{1}{2}\iint_Q \Div (V) u_t^2\df x\df t\\
				&=&\int_\Om u_t(V\cdot\nabla u)\df x\Big|_0^T+\f{1}{2}\iint_Q \Div (V) u_t^2\df x\df t,
			\end{eqnarray*}
			leading to
			\begin{equation*}
				|H_1|\leq C\left(\|u_0\|_{\mcH_0^1(\Om)}+\|u_1\|_{L^2(\Om)}+\|\chi_\omega f\|_{L^2(Q)}\right)^2
			\end{equation*}
			according to \eqref{10.02.3}. We also compute $H_2$ as
			\begin{eqnarray*}
				H_2
				&&=-\iint_Q(V\cdot \nabla u)\Div (|x|^\al \nabla u)\df x\df t\\
				&&=-\iint_Q \Div\Big((V\cdot \nabla u)|x|^\al \nabla u\Big)\df x\df t+\iint_Q |x|^\al \nabla u\cdot \nabla (V\cdot \nabla u)\df x\df t\\
				&&=-\iint_\Sigma |x|^\al |\pt_\nu u|^2\df x\df t+\iint_Q |x|^\al (DV\nabla u)\cdot \nabla u\df x\df t+\f{1}{2}\iint_Q |x|^\al V\cdot \nabla |\nabla u|^2\df x\df t\\
				&&=-\iint_\Sigma |x|^\al |\pt_\nu u|^2\df x\df t+\iint_Q |x|^\al (DV\nabla u)\cdot \nabla u\df x\df t\\
				&&\hspace{4.5mm}+\f{1}{2}\iint_Q \Div\Big(|x|^\al V  |\nabla u|^2\Big)\df x\df t-\f{1}{2}\iint_Q \Div(|x|^\al V)|\nabla u|^2\df x\df t\\
				&&=-\f{1}{2}\iint_\Sigma |x|^\al |\pt_\nu u|^2\df x\df t+\iint_Q|x|^\al (DV\nabla u)\cdot \nabla u\df x\df t-\f{1}{2}\iint_Q \Div(|x|^\al V)|\nabla u|^2\df x\df t.
			\end{eqnarray*}
			Notice that the following inequalities hold:
			\begin{equation*}
				\iint_Q|x|^\al(DV\nabla u)\cdot \nabla u\df x\df t\leq C\iint_Q|x|^\al \nabla u\cdot \nabla u\df x\df t,
			\end{equation*}
			and
			\begin{equation*}
				\iint_Q\Div(|x|^\al V)|\nabla u|^2\df x\df t\leq \f{C}{\e^2}\iint_Q|x|^\al \nabla u\cdot \nabla u\df x\df t,
			\end{equation*}
			and
			\begin{equation*}
				\iint_Q\chi_\omega f(x,t)(V\cdot \nabla u)\df x\df t\leq C\|\chi_\omega f\|_{L^2(Q)}+\f{C}{\e^\al}\iint_Q|x|^\al \nabla u\cdot \nabla u \df x\df t.
			\end{equation*}
			Consequently, we obtain
			\begin{equation*}
				\begin{split}
					\iint_\Sigma |\pt_\nu u|^2\df S\df t
					&\leq \f{C}{\e^2}\left(\|u_0\|_{\mcH_0^1(\Om)}+\|u_1\|_{L^2(\Om)}+\|\chi_\omega f\|_{L^2(Q)}\right)^2.
				\end{split}
			\end{equation*}
			This implies \eqref{10.02.4}.
		\end{proof}
		
		\begin{remark}\label{R1}
			Assume the conditions of Theorem \ref{T2.2} are satisfied. Let
			\begin{equation*}
				E(t)=\left(\||x|^\al \nabla u(t)\cdot \nabla u(t)\|_{L^1(0,T; L^2(\Om))}+\|u_t(t)\|_{L^2(\Om)}^2\right)^\f{1}{2}
			\end{equation*}
			be as defined in the proof of Theorem \ref{T2.2}. Then there exists a constant $C=C(T, \Om)>0$ such that
			\begin{equation}\label{10.02.5}
				E(t)\leq C\left(E(s)+\|\chi_\omega f\|_{L^2(Q)}\right), \ t,s\in [0,T].
			\end{equation}
			Indeed, if  $s\leq T$, then the same argument used in the proof of Theorem \ref{T2.2} leads to \eqref{10.02.5}. On the other hand, if we replace
			$t$ by $T-t$, we can also obtain \eqref{10.02.5} for the case where $s\leq T$.
			
		\end{remark}

		\section{Carleman estimate}\label{Se4}
		To achieve controllability results, it is necessary to employ a Carleman estimate, which is a fundamental mathematical tool in the realm of partial differential equations. By utilizing appropriately chosen weight functions, we can investigate how external control influences the behavior of a system. The weighted integral aids in amplifying contributions from specific regions within the domain, enabling the extraction of information that may otherwise be difficult to obtain. Carleman estimates are pivotal in establishing unique continuation properties, which are crucial for proving controllability. The primary objective of this section is to derive Carleman estimates, and we will prove Theorem \ref{Carleman2} following.
		
		\begin{proof}[{\rm \bf { Proof of  Theorem \ref{Carleman2}.}}]
			Since $0\in\om$, we assume there exists $\varepsilon >0$ such that  $B(0,3\varepsilon )\s\om$. Let
			\begin{equation*}
				z\in C^1([0,T]; L^2(\Om))\cap C([0,T]; \mcH_0^1(\Om))
			\end{equation*}
			be a solution of the following adjoint system (i.e., \eqref{EQ-2}):
			\begin{equation*}
				\begin{cases}
					\pt_t^2z-\Div\left(|x|^\al \nabla z\right)=F, & \mbox{in } Q,\\
					z(0)=z_0, \pt_tz(0)=z_1, &\mbox{in }\pt Q,\\
					z=0, &\mbox{on }\Sigma
				\end{cases}
			\end{equation*}
			with
			\begin{equation}\label{*}
				\supp z\s \Om\se B(0,\e).
			\end{equation}
			Let
			\begin{equation*}
				\psi(x,t)=|x|^2-\be(t-t_0)^2+\be_0
			\end{equation*}
			with $\be\in (0,\varrho), t_0\in (0,T),\be_0\geq 0$. Define
			\begin{equation}\label{vp}
				\vp(x,t)=e^{\ga\psi(x,t)},
			\end{equation}
			and denote
			\begin{equation*}
				w=e^{s\vp}z
			\end{equation*}
			with $z$ being  a solution of \eqref{EQ-2}. It is straightforward to show that
			\begin{equation*}
				e^{s\vp}z_t=w_t-s\vp_tw, \quad e^{s\vp}z_{tt}=w_{tt}-2s\vp_tw_t+s^2\vp_t^2w-s\vp_{tt}w,
			\end{equation*}
			\begin{equation*}
				e^{s\vp}\nabla z=\nabla w-sw\nabla\vp,
			\end{equation*}
			and
			\begin{equation*}
				e^{s\vp}\Div(|x|^\al \nabla z)=\Div\left(|x|^\al \nabla w\right)-2s|x|^\al \nabla \vp\cdot \nabla w-sw\Div\left(|x|^\al \nabla\vp\right)+s^2|x|^\al w\nabla\vp\cdot \nabla\vp.
			\end{equation*}
			These relations imply that
			\begin{equation*}
				\begin{split}
					e^{s\vp}\left(\pt_t^2z-\Div\left(|x|^\al \nabla z\right)\right)
					&=w_{tt}-\Div\left(|x|^\al \nabla w\right)\\
					&\hspace{4.5mm}+s^2\left(\vp_t^2w-|x|^\al w\nabla\vp\cdot \nabla\vp\right)\\
					&\hspace{4.5mm}+s\left(-2\vp_tw_t-\vp_{tt}w+2|x|^\al \nabla\vp\cdot \nabla w+w\Div(|x|^\al \nabla\vp)\right).
				\end{split}
			\end{equation*}
			Hence,
			\begin{equation*}
				P^+w+P^-w=e^{s\vp}F,
			\end{equation*}
			with
			\begin{equation*}
				P^+w:=\left(w_{tt}-\Div\left(|x|^\al \nabla w\right)\right)+s^2\left(\vp_t^2-|x|^\al \nabla\vp\cdot \nabla\vp\right)w,
			\end{equation*}
			and
			\begin{equation*}
				P^-w:=-2s\left(\vp_tw_t-|x|^\al \nabla\vp\cdot \nabla w\right)-s\left(\vp_{tt}-\Div(|x|^\al \nabla\vp)\right)w.
			\end{equation*}
			Therefore,
			\begin{equation*}
				\begin{split}
					\|e^{s\vp}F\|^2=\|P^+w\|^2+\|P^-w\|^2+2(P^+w, P^-w).
				\end{split}
			\end{equation*}

			Now, we proceed to compute the following expression
			\begin{eqnarray*}
				(P^+w, P^-w)
				&&=-2s\iint_Q w_{tt}(\vp_tw_t-|x|^\al \nabla\vp\cdot \nabla w)\df x\df t\\ 
				&& -s\iint_Qw_{tt}(\vp_{tt}-\Div (|x|^\al \nabla\vp))w\df x\df t\\ 
				&&+2s\iint_Q \Div(|x|^\al \nabla w) \left(\vp_tw_t-|x|^\al \nabla\vp\cdot \nabla w\right)\df x\df t\\ 
				&& +s\iint_Q\Div(|x|^\al \nabla w)\left(\vp_{tt}-\Div (|x|^\al \nabla \vp)\right)w\df x\df t\\ 
				&&-2s^3\iint_Q(\vp_t^2w-|x|^\al w\nabla\vp\cdot \nabla\vp)\left(\vp_t w_t-|x|^\al \nabla\vp\cdot \nabla w\right)\df x\df t\\ 
				&& -s^3\iint_Q\left(\vp_t^2-|x|^\al \nabla\vp\cdot \nabla\vp\right) \left(\vp_{tt}-\Div (|x|^\al \nabla\vp)\right)w^2 \df x\df t\\ 
				&&=:\sum_{i=1}^6 I_i.
			\end{eqnarray*}

			Next, we compute  $I_1$. Given that  $w_t(0)=w_t(T)=0$ and $w|_{\Sigma}=0$, we derive
			\begin{eqnarray*}
				I_1
				&&=-2s\iint_Q w_{tt}(\vp_tw_t-|x|^\al \nabla\vp\cdot \nabla w)\df x\df t \\
				&&=-s\iint_Q \vp_t \pt_t w_t^2\df x\df t+2s\iint_Q |x|^\al \nabla\vp\cdot \nabla w w_{tt}\df x\df t\\
				&&=-s\iint_Q\pt_t\left(\vp_t w_t^2\right)\df x\df t+s\iint_Q \vp_{tt}w_t^2\df x\df t\\
				&& +2s\iint_Q\pt_t\left(|x|^\al \nabla\vp\cdot \nabla w w_t\right)\df x\df t-2s\iint_Q|x|^\al\nabla\vp_t \cdot \nabla w w_t\df x\df t-2s\iint_Q |x|^\al\nabla \vp\cdot \nabla w_t w_t\df x\df t\\
				&&=s\iint_Q \vp_{tt}w_t^2\df x\df t-2s\iint_Q|x|^\al(\nabla\vp_t \cdot \nabla w) w_t\df x\df t\\
				&& -s\iint_Q \Div\left(|x|^\al\nabla \vp    w_t^2\right) \df x\df t+s\iint_Q \Div\left(|x|^\al \nabla \vp\right) w_t^2\df x\df t\\
				&&=s\iint_Q \vp_{tt}w_t^2\df x\df t-2s\iint_Q|x|^\al(\nabla\vp_t \cdot \nabla w) w_t\df x\df t+s\iint_Q \Div\left(|x|^\al \nabla \vp\right) w_t^2\df x\df t\\
				&&=s\iint_Q\left(\vp_{tt}+\Div \left(|x|^\al \nabla\vp\right)\right) w_t^2\df x\df t-2s\iint_Q|x|^\al(\nabla\vp_t \cdot \nabla w) w_t\df x\df t.
			\end{eqnarray*}
			Now, let's compute $I_2$.  Given that $w(0)=w(T)=0$, we obtain
			\begin{eqnarray*}
				I_2
				&&=-s\iint_Q(\vp_{tt}-\Div (|x|^\al \nabla\vp))w_{tt}w\df x\df t\\
				&&=-s\iint_Q\pt_t\Big( \left(\vp_{tt}-\Div (|x|^\al \nabla\vp)\right)w_{t}w\Big)\df x\df t\\
				&& +s\iint_Q\left(\vp_{tt}-\Div (|x|^\al \nabla \vp)\right)w_t^2\df x\df t+\f{s}{2}\iint_Q \pt_t w^2\left(\pt_t^2\vp_t-\Div(|x|^\al \nabla \vp_t)\right)\df x\df t\\
				&&=s\iint_Q\left(\vp_{tt}-\Div (|x|^\al \nabla \vp)\right)w_t^2\df x\df t\\
				&&\hspace{4.5mm}+\f{s}{2}\iint_Q\pt_t\Big(w^2(\pt_t^2\vp_t-\Div(|x|^\al \nabla \vp_t))\Big)\df x\df t-\f{s}{2}\iint_Q w^2\left(\pt_{tt}^2\vp-\Div(|x|^\al \nabla \vp_{tt})\right)\df x\df t\\
				&&=s\iint_Q\left(\vp_{tt}-\Div (|x|^\al \nabla \vp)\right)w_t^2\df x\df t-\f{s}{2}\iint_Q w^2\left(\pt_{tt}^2\vp-\Div(|x|^\al \nabla \vp_{tt})\right)\df x\df t,
			\end{eqnarray*}
			where we have used the fact that the integral of a total derivative over a time interval $[0,T]$ with zero initial and final conditions vanishes.

			Next, we proceed to compute  $I_3$. Given the condition  $w|_{\Sigma}=0$,  we derive the following:
			\begin{eqnarray*}
				I_3
				&&=2s\iint_Q \Div(|x|^\al \nabla w) \left(\vp_tw_t-|x|^\al \nabla\vp\cdot \nabla w\right)\df x\df t\\
				&&=2s\iint_Q\Div\Big(|x|^\al \nabla w(\vp_tw_t-|x|^\al \nabla\vp\cdot \nabla w)\Big)\df x\df t-2s\iint_Q|x|^\al \nabla w\cdot \nabla (\vp_tw_t-|x|^\al \nabla\vp\cdot \nabla w)\df x\df t\\
				&&=-2s\iint_\Sigma |x|^{2\al}(\nabla w\cdot \nu)(\nabla \vp\cdot \nabla w)\df S\df t\\
				&&\hspace{4.5mm}-2s\iint_Q|x|^\al (\nabla w\cdot \nabla \vp_t)w_t\df x\df t-2s\iint_Q|x|^\al \vp_t\nabla w\cdot \nabla w_t\df x\df t\\
				&&\hspace{4.5mm}+2s\iint_Q|x|^{2\al-2}(\nabla w\cdot x)(\nabla \vp\cdot \nabla w)\df x\df t+2s\iint_Q|x|^{2\al}\nabla w\cdot \nabla(\nabla \vp\cdot \nabla w)\df x\df t\\
				&&=-s\iint_\Sigma  |x|^{2\al}(\nabla \vp\cdot\nu) |\pt_\nu w|^2 \df S\df t+2s\iint_Q|x|^{2\al-2}(\nabla w\cdot x)(\nabla \vp\cdot \nabla w)\df x\df t\\
				&&\hspace{4.5mm}-2s\iint_Q|x|^\al (\nabla w\cdot \nabla \vp_t)w_t\df x\df t+s\iint_Q|x|^\al \vp_{tt}(\nabla w\cdot \nabla w)\df x\df t\\
				&&\hspace{4.5mm}+2s\iint_Q|x|^{2\al}D^2\vp\nabla w\cdot \nabla w\df x\df t-s\iint_Q\Div(|x|^{2\al}\nabla\vp)|\nabla w|^2\df x\df t
			\end{eqnarray*}
			where we have used the identities:
			\begin{equation*}
				\begin{split}
					&-2s\iint_Q|x|^\al \vp_t\nabla w\cdot \nabla\vp_t\df x\df t\\
					&=-s\iint_Q|x|^\al\vp_t\pt_t\left(\nabla w\cdot \nabla w\right)\df x\df t\\
					&=-s\iint_Q\pt_t\left(|x|^\al \vp_t (\nabla w\cdot \nabla w)\right)\df x\df t+s\iint_Q|x|^\al \vp_{tt}(\nabla w\cdot \nabla w)\df x\df t\\
					&=s\iint_Q|x|^\al \vp_{tt}(\nabla w\cdot \nabla w)\df x\df t,
				\end{split}
			\end{equation*}
			and
			\begin{eqnarray*}
				&&2s\iint_Q|x|^{2\al}\nabla w\cdot \nabla(\nabla \vp\cdot \nabla w)\df x\df t\\
				&&=2s\iint_Q|x|^{2\al}D^2\vp\nabla w\cdot \nabla w\df x\df t+s\iint_Q|x|^{2\al}\nabla \vp\cdot \nabla (|\nabla w|^2)\df x\df t\\
				&&=2s\iint_Q|x|^{2\al}D^2\vp\nabla w\cdot \nabla w\df x\df t\\
				&&\hspace{4.5mm}+s\iint_Q\Div\left( |x|^{2\al}\nabla \vp |\nabla w|^2\right)\df x\df t-s\iint_Q\Div(|x|^{2\al}\nabla\vp)|\nabla w|^2\df x\df t\\
				&&=2s\iint_Q|x|^{2\al}D^2\vp\nabla w\cdot \nabla w\df x\df t\\
				&&\hspace{4.5mm}+s\iint_\Sigma  |x|^{2\al}(\nabla \vp\cdot\nu) |\pt_\nu w|^2 \df S\df t-s\iint_Q\Div(|x|^{2\al}\nabla\vp)|\nabla w|^2\df x\df t\\
			\end{eqnarray*}
			with $D^2\vp=(\pt_{x_ix_j}^2\vp)_{i,j\in J_N}$ and
			\begin{equation*}
				\nabla w\cdot \nabla (\nabla w\cdot \nabla \vp)=D^2\vp\nabla w\cdot \nabla w+\f{1}{2}\nabla \vp\cdot \nabla (|\nabla z|^2).
			\end{equation*}

			Now, let us compute  $I_4$. Given that  $w|_{\Sigma}=0$, we obtain
			\begin{eqnarray*}
				I_4
				&&=s\iint_Q\Div(|x|^\al \nabla w)\left(\vp_{tt}-\Div (|x|^\al \nabla \vp)\right)w\df x\df t\\
				&&=s\iint_Q\Div\Big(|x|^\al \nabla w\left(\vp_{tt}-\Div (|x|^\al \nabla \vp)\right)w\Big)\df x\df t-s\iint_Q|x|^\al \nabla w\cdot \nabla \Big( \left(\vp_{tt}-\Div (|x|^\al \nabla \vp)\right)w\Big)\df x\df t\\
				&&=-s\iint_Q|x|^\al \nabla w\cdot \nabla w(\vp_{tt}-\Div(|x|^\al \nabla \vp))\df x\df t-\f{s}{2}\iint_Q|x|^\al  \nabla (\vp_{tt}-\Div (|x|^\al \nabla\vp))\cdot \nabla w^2\df x\df t\\
				&&=-s\iint_Q|x|^\al \nabla w\cdot \nabla w(\vp_{tt}-\Div(|x|^\al \nabla \vp))\df x\df t\\
				&&\hspace{4.5mm}-\f{s}{2}\iint_Q\Div\Big(  \nabla (\vp_{tt}-\Div (|x|^\al \nabla\vp))w^2\Big)\df x\df t+\f{s}{2}\iint_Q\Div\Big(|x|^\al \nabla  \left(\vp_{tt}-\Div (|x|^\al \nabla \vp)\right) \Big) w^2\df x\df t\\
				&&=-s\iint_Q|x|^\al \nabla w\cdot \nabla w(\vp_{tt}-\Div(|x|^\al \nabla \vp))\df x\df t+\f{s}{2}\iint_Q\Div\Big(|x|^\al \nabla  \left(\vp_{tt}-\Div (|x|^\al \nabla \vp)\right) \Big) w^2\df x\df t.
			\end{eqnarray*}

			Now, we proceed to compute $I_5$. Given  $w(0)=w(T)=0$ and $ w|_{\Sigma}=0$, we derive:
			\begin{eqnarray*}
				I_5
				&&=-2s^3\iint_Q(\vp_t^2w-|x|^\al w\nabla\vp\cdot \nabla\vp)\left(\vp_t w_t-|x|^\al \nabla\vp\cdot \nabla w\right)\df x\df t\\
				&&=-2s^3\iint_Q\vp_t^3w_tw\df x\df t+2s^3\iint_Q|x|^\al \vp_t^2\nabla \vp\cdot \nabla w w\df x\df t\\
				&&\hspace{4.5mm}+2s^3\iint_Q |x|^\al \vp_t(\nabla \vp\cdot \nabla \vp)w_t w\df x\df t-2s^3\iint_Q |x|^{2\al}(\nabla\vp\cdot \nabla\vp)(\nabla \vp\cdot \nabla w)w\df x\df t \\
				&&=-s^3\iint_Q\vp_t^3\pt_tw^2\df x\df t+s^3\iint_Q|x|^\al\vp_t^2\nabla \vp\cdot \nabla w^2\df x\df t\\
				&&\hspace{4.5mm}+s^3\iint_Q|x|^\al \vp_t(\nabla\vp\cdot \nabla\vp)\pt_tw^2\df x\df t-s^3\iint_Q|x|^{2\al}(\nabla\vp\cdot \nabla\vp)\nabla\vp\cdot \nabla w^2\df x\df t\\
				&&=-s^3\iint_Q\pt_t\left(\vp_t^3w^2\right)\df x\df t+3s^3\iint_Q \vp_t^2\vp_{tt}w^2\df x\df t\\
				&&\hspace{4.5mm}+s^3\iint_Q\Div\Big(|x|^\al \vp_t^2\nabla \vp w\Big)\df x\df t-s^3\iint_Q\Div(|x|^\al \vp_t^2\nabla \vp)w^2\df x\df t\\
				&&\hspace{4.5mm}+s^3\iint_Q\pt_t\Big(|x|^\al\vp_t(\nabla\vp\cdot \nabla \vp) w\Big)\df x\df t-s^3\iint_Q\pt_t(|x|^\al \vp_t(\nabla \vp\cdot \nabla \vp))w^2\df x\df t\\
				&&\hspace{4.5mm}-s^3\iint_Q\Div\Big(|x|^{2\al}(\nabla\vp\cdot \nabla\vp)\nabla\vp w^2\Big)\df x\df t+s^3\iint_Q\Div\left(|x|^{2\al}(\nabla\vp\cdot \nabla\vp)\nabla\vp\right)w^2\df x\df t\\
				&&7=3s^3\iint_Q \vp_t^2\vp_{tt}w^2\df x\df t-s^3\iint_Q\Div(|x|^\al \vp_t^2\nabla \vp)w^2\df x\df t\\
				&&\hspace{4.5mm}-s^3\iint_Q\pt_t(|x|^\al \vp_t(\nabla \vp\cdot \nabla \vp))w^2\df x\df t+s^3\iint_Q\Div\left(|x|^{2\al}(\nabla\vp\cdot \nabla\vp)\nabla\vp\right)w^2\df x\df t.
			\end{eqnarray*}
			
			Next, we compute $I_6$:
			\begin{equation*}
				\begin{split}
					I_6=&-s^3\iint_Q\left(\vp_t^2-|x|^\al \nabla\vp\cdot \nabla\vp\right) \left(\vp_{tt}-\Div (|x|^\al \nabla\vp)\right)w^2 \df x\df t.
				\end{split}
			\end{equation*}
			
			Combining  $I_1$ to $I_6$, we obtain
			\begin{eqnarray*}
				&&(P^+w,P^-w)\\
				&&=s\iint_Q\left(\vp_{tt}+\Div \left(|x|^\al \nabla\vp\right)\right) w_t^2\df x\df t-2s\iint_Q|x|^\al(\nabla\vp_t \cdot \nabla w) w_t\df x\df t\\ 
				&&\hspace{4.5mm}+s\iint_Q\left(\vp_{tt}-\Div (|x|^\al \nabla \vp)\right)w_t^2\df x\df t-\f{s}{2}\iint_Q w^2\left(\pt_{tt}^2\vp -\Div(|x|^\al \nabla \vp_{tt})\right)\df x\df t\\ 
				&&\hspace{4.5mm}-s\iint_\Sigma  |x|^{2\al}(\nabla \vp\cdot\nu) |\pt_\nu w|^2 \df S\df t+2s\iint_Q|x|^{2\al-2}(\nabla w\cdot x)(\nabla \vp\cdot \nabla w)\df x\df t\\
				&&\hspace{4.5mm}-2s\iint_Q|x|^\al (\nabla w\cdot \nabla \vp_t)w_t\df x\df t+s\iint_Q|x|^\al \vp_{tt}(\nabla w\cdot \nabla w)\df x\df t\\
				&&\hspace{4.5mm}+2s\iint_Q|x|^{2\al}D^2\vp\nabla w\cdot \nabla w\df x\df t-s\iint_Q\Div(|x|^{2\al}\nabla\vp)|\nabla w|^2\df x\df t\\ 
				&&\hspace{4.5mm} -s\iint_Q|x|^\al \nabla w\cdot \nabla w(\vp_{tt}-\Div(|x|^\al \nabla \vp))\df x\df t+\f{s}{2}\iint_Q\Div\Big(|x|^\al \nabla  \left(\vp_{tt}-\Div (|x|^\al \nabla \vp)\right) \Big) w^2\df x\df t \\ 
				&&\hspace{4.5mm} +3s^3\iint_Q \vp_t^2\vp_{tt}w^2\df x\df t-s^3\iint_Q\Div(|x|^\al \vp_t^2\nabla \vp)w^2\df x\df t\\
				&&\hspace{4.5mm}-s^3\iint_Q\pt_t(|x|^\al \vp_t(\nabla \vp\cdot \nabla \vp))w^2\df x\df t+s^3\iint_Q\Div\left(|x|^{2\al}(\nabla\vp\cdot \nabla\vp)\nabla\vp\right)w^2\df x\df t\\ 
				&&\hspace{4.5mm} -s^3\iint_Q\left(\vp_t^2-|x|^\al \nabla\vp\cdot \nabla\vp\right) \left(\vp_{tt}-\Div (|x|^\al \nabla\vp)\right)w^2 \df x\df t\\
				&&=\sum_{i=1}^{17} H_i.
			\end{eqnarray*}
			
			Now, we estimate $(P^+w,P^-w)$. The fifth term $H_5$  does not need estimation:
			\begin{equation*}
				-s\iint_\Sigma  |x|^{2\al}(\nabla \vp\cdot\nu) |\pt_\nu w|^2 \df S\df t =-2s\ga \iint_\Sigma |x|^{2\al}  \vp (x\cdot \nu)|\pt_\nu w|^2\df S\df t.
			\end{equation*}
			
			Next, we estimate the sixth term $H_6$:
			\begin{eqnarray*}
				&&2s\iint_Q|x|^{2\al-2}(\nabla w\cdot x)(\nabla \vp\cdot \nabla w)\df x\df t\\
				&&=2s\iint_Q |x|^{2\al-2}(\nabla w\cdot x) (2\ga \vp x\cdot \nabla w)\df x\df t\\
				&&=4s\ga\iint_Q|x|^{2\al-2}\vp (x\cdot \nabla w)^2\df x\df t.
			\end{eqnarray*}
			
			Now, we estimate the tenth term $H_{9}$:
			\begin{eqnarray*}
				H_{9}
				&&=2s\iint_Q|x|^{2\al}D^2\vp\nabla w\cdot \nabla w\df x\df t\\
				&&=2s\iint_Q|x|^{2\al} \vp\left(2\ga \de_{ij}+4\ga^2x_ix_j\right) \nabla w\cdot \nabla w\df x\df t\\
				&&=2s \alpha \ga \iint_Q|x|^{2\al}\vp |\nabla w|^2\df x\df t+8s\ga^2\iint_Q|x|^{2\al}(x\cdot \nabla w)^2\df x\df t.
			\end{eqnarray*}

			Next, we estimate the first term $H_1$:
			\begin{eqnarray*}
				H_1
				&&=s\iint_Q\left(\vp_{tt}+\Div \left(|x|^\al \nabla\vp\right)\right) w_t^2\df x\df t\\
				&&=s\iint_Q\vp(-2\be\ga+\ga^2\psi_t^2)w_t^2\df x\df t+s\iint_Q\Div(|x|^\al \nabla \vp)w_t^2\df x\df t\\
				&&\geq-2s\be\ga\iint_Q\vp w_t^2\df x\df t+s\ga^2\iint_Q \psi_t^2\vp w_t^2\df x\df t+2\ga s\iint_Q\vp\left(N+1+2\ga |x|^2\right)|x|^\al w_t^2\df x\df t.
			\end{eqnarray*}

			Now, we estimate the second term $H_2$:
			\begin{eqnarray*}
				H_2
				&&=-2s\iint_Q|x|^\al(\nabla\vp_t \cdot \nabla w) w_t\df x\df t\\
				&&=-2s\ga^2\iint_Q |x|^\al \vp  \psi_t\nabla \psi \cdot \nabla w w_t\df x\df t=-2s\ga^2\iint_Q( |x|^\al \vp^\f{1}{2}x\cdot \nabla w)(\vp^\f{1}{2}\psi_tw_t)\df x\df t \\
				&&\geq -s\ga^2\iint_Q|x|^{2\al}\vp(x\cdot \nabla w)^2\df x\df t-s\ga^2 \iint_Q\psi_t^2 \vp w_t^2\df x\df t.
			\end{eqnarray*}

			Now, we estimate $H_3$:
			\begin{eqnarray*}
				H_3
				&& =s\iint_Q\left(\vp_{tt}-\Div (|x|^\al \nabla \vp)\right)w_t^2\df x\df t\\
				&&\geq -2s\be\ga \iint_Q\vp w_t^2\df x\df t+s\ga^2\iint_Q \psi_t^2\vp w_t^2\df x\df t-2\ga s\iint_Q\vp\left(N+1+2\ga |x|^2\right)|x|^\al w_t^2\df x\df t.
			\end{eqnarray*}

			Now, we estimate $H_7$:
			\begin{equation*}
				H_7
				=-2s\iint_Q|x|^\al (\nabla w\cdot \nabla \vp_t)w_t\df x\df t=H_2.
			\end{equation*}
			
			Next, we estimate  $H_8+H_{10}+H_{11}$:
			\begin{eqnarray*}
				&&H_8+H_{10}+H_{11}\\
				&&=s\iint_Q|x|^\al \vp_{tt}(\nabla w\cdot \nabla w)\df x\df t -s\iint_Q\Div(|x|^{2\al}\nabla\vp)|\nabla w|^2\df x\df t\\
				&&\hspace{4.5mm}-s\iint_Q|x|^\al \nabla w\cdot \nabla w(\vp_{tt}-\Div(|x|^\al \nabla \vp))\df x\df t\\
				&&=-s\iint_Q \nabla (|x|^\al) |x|^\al \nabla \vp|\nabla w|^2\df x\df t=-2s \alpha\ga\iint_Q|x|^{2\al}\vp|\nabla w|^2\df x\df t.
			\end{eqnarray*}

			Now, we obtain that
			\begin{equation*}
				\begin{split}
					& H_1+H_2+H_3+H_5+H_6+H_7+H_8+H_9+H_{10}+H_{11}\\
					&\geq 2(2-\alpha)s\ga \iint_Q|x|^{2\al}\vp |\nabla w|^2\df x\df t+6s\ga^2\iint_Q|x|^{2\al}(x\cdot \nabla w)^2\df x\df t\\
					&\hspace{4.5mm}+4s\ga\iint_Q|x|^{2\al-2}\vp (x\cdot \nabla w)^2\df x\df t \\
					&\hspace{4.5mm}-2s\ga \iint_\Sigma |x|^{2\al}  \vp (x\cdot \nu)|\pt_\nu w|^2\df S\df t -4s\be\ga\iint_Q\vp w_t^2\df x\df t.
				\end{split}
			\end{equation*}

			Now, we compute
			\begin{eqnarray*}
				&&(P^+w, s\ga \vp w)\\
				&& =\Big(\left(w_{tt}-\Div\left(|x|^\al \nabla w\right)\right)+s^2\left(\vp_t^2-|x|^\al \nabla\vp\cdot \nabla\vp\right)w, s\ga \vp w\Big)\\
				&&=s\ga \iint_Q \vp w_{tt} w\df x\df t-s\ga \iint_Q \vp \Div(|x|^\al \nabla w)w\df x\df t\\
				&&\hspace{4.5mm} +s^3\ga \iint_Q \vp_t^2\vp w^2\df x\df t-s^3\ga \iint_Q |x|^\al \vp (\nabla\vp\cdot \nabla\vp) w^2\df x\df t\\
				&&=\sum_{i=1}^4L_i.
			\end{eqnarray*}

			Now, we compute $L_1$. Since $w(0)=w(T)=0$, we obtain
			\begin{eqnarray*}
				L_1
				&&=s\ga \iint_Q \vp w_{tt} w\df x\df t\\
				&&=s\ga\iint_Q \pt_t(\vp w_tw)\df x\df t-s\ga\iint_Q \vp w_t^2\df x\df t-s\ga \iint_Q\vp_tw_tw\df x\df t\\
				&&=-s\ga\iint_Q\vp w_t^2\df x\df t -\f{s\ga }{2}\iint_Q\vp_t\pt_tw^2\df x\df t\\
				&&=-s\ga\iint_Q\vp w_t^2\df x\df t-\f{s\ga }{2}\iint_Q \pt_t(\vp_tw^2)\df x\df t+\f{s\ga }{2}\iint_Q \vp_{tt}w^2\df x\df t\\
				&&=-s\ga\iint_Q\vp w_t^2\df x\df t+\f{s\ga }{2}\iint_Q \vp_{tt}w^2\df x\df t.
			\end{eqnarray*}

			Now, we compute $L_2$:
			\begin{equation*}
				\begin{split}
					L_2
					&=-s\ga \iint_Q \vp \Div(|x|^\al \nabla w)w\df x\df t\\
					&=-s\ga \iint_Q\Div\left(|x|^\al \vp \nabla w w\right)\df x\df t+s\ga\iint_Q|x|^\al \nabla\vp\cdot \nabla w w\df x\df t+s\ga\iint_Q|x|^\al \vp \nabla w\cdot \nabla w\df x\df t\\
					&=\f{s\ga}{2}\iint_Q|x|^\al \nabla \vp\cdot \nabla w^2\df x\df t+s\ga \iint_Q|x|^\al \vp\nabla w\cdot \nabla w\df x\df t\\
					&=\f{s\ga}{2}\iint_Q\Div\left(|x|^\al \nabla \vp w^2\right)\df x\df t-\f{s\ga}{2}\iint_Q\Div\left(|x|^\al \nabla \vp\right)w^2\df x\df t+s\ga\iint_Q|x|^\al \vp \nabla w\cdot \nabla w\df x\df t\\
					&=-\f{s\ga}{2}\iint_Q\Div\left(|x|^\al \nabla \vp\right)w^2\df x\df t+s\ga\iint_Q|x|^\al \vp \nabla w\cdot \nabla w\df x\df t.
				\end{split}
			\end{equation*}
			
			Now, we compute $L_3+L_4$:
			\begin{eqnarray*}
				&&L_3+L_4\\
				&&=s^3\ga \iint_Q \vp_t^2\vp w^2\df x\df t-s^3\ga^3 \iint_Q |x|^\al \vp (\nabla\vp\cdot \nabla\vp) w^2\df x\df t=s^3\ga^3\iint_Q \vp^3(\psi_t^2-4|x|^{\al+2})w^2\df x\df t.
			\end{eqnarray*}
			Hence,
			\begin{eqnarray*}
				&&(P^+w, s\ga \vp w)\\
				&&=-s\ga\iint_Q\vp w_t^2\df x\df t+\f{s\ga }{2}\iint_Q \Big(\vp_{tt}-\Div\left(|x|^\al \nabla \vp\right)\Big)w^2\df x\df t\\
				&&\hspace{4.5mm} +s\ga\iint_Q|x|^\al \vp \nabla w\cdot \nabla w\df x\df t +s^3\ga^3\iint_Q \vp^3(\psi_t^2-4|x|^{\al+2})w^2\df x\df t,
			\end{eqnarray*}
			which implies that
			\begin{eqnarray*}
				s\ga \iint_Q\vp w_t^2\df x\df t
				&&\leq s^3\ga^3\iint_Q\vp^3\left|\psi_t^2-4|x|^{\al+2}\right| w^2\df x\df t+\f{s\ga}{\e^{\al}}\iint_Q |x|^{2\al}\vp\nabla w\cdot \nabla w\df x\df t\\
				&&\hspace{4.5mm}+\eta\|P^+w\|^2+\f{s^2\ga^3}{2\eta}\iint_Q  \vp^2w^2\df x\df t,
			\end{eqnarray*}
			i.e.,
			\begin{eqnarray*}
				&& s\ga \iint_Q|x|^{2\al}\vp|\nabla w|^2\df x\df t\\
				&&\geq s\ga \e^\al \iint_Q\vp w_t^2\df x\df t-\e^\al\left(s^3\ga^3\iint_Q\vp^3\left|\psi_t^2-4|x|^{\al+2}\right|w^2\df x\df t-\eta\|P^+z\|^2+\f{s^2\ga^2}{2\eta}\iint_Q\vp^2w^2\df x\df t\right).
			\end{eqnarray*}
			Then we get
			\begin{equation*}
				\begin{split}
					& H_1+H_2+H_3+H_5+H_6+H_7+H_8+H_9+H_{10}+H_{11}\\
					&\geq (2-\alpha)s\ga \iint_Q|x|^{2\al}\vp |\nabla w|^2\df x\df t+6s\ga^2\iint_Q|x|^{2\al}(x\cdot \nabla w)^2\df x\df t\\
					&\hspace{4.5mm}+4s\ga\iint_Q|x|^{2\al-2}\vp (x\cdot \nabla w)^2\df x\df t +s\ga \be\iint_Q\vp w_t^2\df x\df t\\
					&\hspace{4.5mm}-2s\ga \iint_\Sigma |x|^{2\al}  \vp (x\cdot \nu)|\pt_\nu w|^2\df S\df t\\
					&\hspace{4.5mm}-(2-\alpha)\e^\al\left(s^3\ga^3\iint_Q\vp^3\left|\psi_t^2-4|x|^{\al+2}\right|w^2\df x\df t-\eta\|P^+z\|^2+\f{s^2\ga^3}{2\eta}\iint_Q\vp^2w^2\df x\df t\right).
				\end{split}
			\end{equation*}
			by choosing  $0<\be<\f{2-\alpha}{5}\e^\al$.

			Now, we estimate $H_4+H_{12}$.  It is straightforward to derive that
			\begin{eqnarray*}
				&& -\f{s}{2}\iint_Q w^2\left(\pt_{tt}^2\vp -\Div(|x|^\al \nabla \vp_{tt})\right)\df x\df t+\f{s}{2}\iint_Q\Div\Big(|x|^\al \nabla  \left(\vp_{tt}-\Div (|x|^\al \nabla \vp)\right) \Big) w^2\df x\df t\\
				&&\geq -\f{s\ga^4}{2}\iint_Q(\psi_t^2-4|x|^{\al+2})^2\vp w^2\df x\df t-Cs\ga^3T^3\iint_Q \vp w^2\df x\df t.
			\end{eqnarray*}

			Next, we estimate  $H_{13}+H_{14}+H_{15}+H_{16}+H_{17}$:
			\begin{eqnarray*}
				&& H_{13}+H_{14}+H_{15}+H_{16}+H_{17}\\
				&& =+3s^3\iint_Q \vp_t^2\vp_{tt}w^2\df x\df t-s^3\iint_Q\Div(|x|^\al \vp_t^2\nabla \vp)w^2\df x\df t\\
				&&\hspace{4.5mm}-s^3\iint_Q\pt_t(|x|^\al \vp_t(\nabla \vp\cdot \nabla \vp))w^2\df x\df t+s^3\iint_Q\Div\left(|x|^{2\al}(\nabla\vp\cdot \nabla\vp)\nabla\vp\right)w^2\df x\df t\\
				&&\hspace{4.5mm} -s^3\iint_Q\left(\vp_t^2-|x|^\al \nabla\vp\cdot \nabla\vp\right) \left(\vp_{tt}-\Div (|x|^\al \nabla\vp)\right)w^2 \df x\df t\\
				&&=2s^3\iint_Q\vp_t^2\vp_{tt}w^2\df x\df t-4s^3\iint_Q|x|^\al \vp_t\nabla\vp_t\cdot \nabla \vp w^2\df x\df t\\
				&&\hspace{4.5mm}+s^3\iint_Q |x|^\al \nabla \vp\cdot \nabla \left(|x|^\al \nabla \vp\cdot \nabla \vp\right)w^2\df x\df t\\
				&&=2s^3\ga^3\iint_Q\vp^3\psi_t^2\psi_{tt}w^2\df x\df t+2s^3\ga^4\iint_Q\vp^3\psi_t^4w^2\df x\df t-16s^3\ga^4\iint_Q|x|^{\al+2}\psi_t^2\vp^3w^2\df x\df t\\
				&&\hspace{4.5mm}+8(\al+2)s^3\ga^3\iint_Q|x|^{2\al+2}\vp^3w^2\df x\df t+32s^3\ga^4\iint_Q|x|^{2\al+4}\vp^3w^2\df x\df t\\
				&&=-4s^3\ga^3\be\iint_Q (\psi_t^2-4|x|^{\al+2})\vp^3w^2\df x\df t+2s^3\ga^4\iint_Q\left(\psi_t^2-4|x|^{\al+2}\right)^2\vp^3w^2\df x\df t\\
				&&\hspace{4.5mm}+8(\al+2)s^3\ga^3\iint_Q|x|^{2\al+2}\vp^3w^2\df x\df t-16s^3\ga^3\be\iint_Q |x|^{\al+2}\vp^3w^2\df x\df t.
			\end{eqnarray*}

			Overall, we obtain
			\begin{eqnarray*}
				&&(P^+w,P^-w)\\
				&&\geq (2-\alpha)s\ga \iint_Q|x|^{2\al}\vp |\nabla w|^2\df x\df t+6s\ga^2\iint_Q|x|^{2\al}(x\cdot \nabla w)^2\df x\df t\\
				&&\hspace{4.5mm}+4s\ga\iint_Q|x|^{2\al-2}\vp (x\cdot \nabla w)^2\df x\df t +s\ga \be\iint_Q\vp w_t^2\df x\df t\\
				&&\hspace{4.5mm}-2s\ga \iint_\Sigma |x|^{2\al}  \vp (x\cdot \nu)|\pt_\nu w|^2\df S\df t\\
				&&\hspace{4.5mm}-(2-\alpha)\e^\al\left(s^3\ga^3\iint_Q\vp^3\left|\psi_t^2-4|x|^{\al+2}\right|w^2\df x\df t-\eta\|P^+z\|^2+\f{s^2\ga^3}{2\eta}\iint_Q\vp^2w^2\df x\df t\right)\\
				&&\hspace{4.5mm} -\f{s\ga^4}{2}\iint_Q(\psi_t^2-4|x|^{\al+2})^2\vp w^2\df x\df t-\f{Cs\ga^3T^3}{\e^{2\al+2}}\iint_Q |x|^{2\al+2} \vp w^2\df x\df t\\
				&&\hspace{4.5mm}-4s^3\ga^3\be\iint_Q \left|\psi_t^2-4|x|^{\al+2}\right|\vp^3w^2\df x\df t+2s^3\ga^4\iint_Q\left(\psi_t^2-4|x|^{\al+2}\right)^2\vp^3w^2\df x\df t\\
				&&\hspace{4.5mm}+8(\al+2)s^3\ga^3\iint_Q|x|^{2\al+2}\vp^3w^2\df x\df t-16s^3\ga^3\be\iint_Q |x|^{\al+2}\vp^3w^2\df x\df t.
				3	
			\end{eqnarray*}
			Now, choose $\eta=\f{1}{4\e^\al}$. Note that
			\begin{eqnarray*}
				&&2s^3\ga^4\iint_Q(\psi_t^2-4|x|^{\al+2})^2\vp^3w^2\df x\df t-(\e^\al +4\be)s^3\ga^3\iint_Q\vp^3\left|\psi_t^2-4|x|^{\al+2}\right| w^2\df x\df t\\
				&&=2s^3\ga^4\iint_Q(\psi_t^2-4|x|^{\al+2})^2\vp^3w^2\df x\df t- \f{(\e^\al+4\be)^2}{4\e^{2\al+2}}s^3\ga^3\iint_Q\left(\psi_t^2-4|x|^{\al+2}\right)^2\vp^3w^2\df x\df t\\
				&&\hspace{4.5mm} +\iint_Q\left[\f{\e^\al+4\be}{2\e^{\al+1}}\left|\psi_t^2-4|x|^{\al+2}\right|-\e^{\al+1}\right]^2\vp^3w^2\df x\df t-\e^{2\al+2}s^3\ga^3\iint_Q\vp^3w^2\df x\df t\\
				&&\geq 2s^3\ga^4\iint_Q(\psi_t^2-4|x|^{\al+2})^2\vp^3w^2\df x\df t- \f{(\e^\al+4\be)^2}{4\e^{2\al+2}}s^3\ga^3\iint_Q\left(\psi_t^2-4|x|^{\al+2}\right)^2\vp^3w^2\df x\df t\\
				&&\hspace{4.5mm} +\iint_Q\left[\f{\e^\al+4\be}{2\e^{\al+1}}\left|\psi_t^2-4|x|^{\al+2}\right|-\e^{\al+1}\right]^2\vp^3w^2\df x\df t- s^3\ga^3\iint_Q|x|^{2\al+2}\vp^3w^2\df x\df t.
			\end{eqnarray*}
			Select $\ga^*=\ga_*(\e)>1, s_*(\ga)>1$   sufficiently large, then for all $\ga \geq \ga_*, s\geq s_*(\ga)$, we obtain
			\begin{eqnarray*}
				&&(P^+w,P^-w)\\
				&&\geq (2-\alpha)s\ga \iint_Q|x|^{2\al}\vp |\nabla w|^2\df x\df t+6s\ga^2\iint_Q|x|^{2\al}(x\cdot \nabla w)^2\df x\df t\\
				&&\hspace{4.5mm}+4s\ga\iint_Q|x|^{2\al-2}\vp (x\cdot \nabla w)^2\df x\df t +s\ga \be\iint_Q\vp w_t^2\df x\df t \\
				&&\hspace{4.5mm}-2s\ga \iint_\Sigma |x|^{2\al}  \vp (x\cdot \nu)|\pt_\nu w|^2\df S\df t-\f{1}{4}\|P^+z\|^2\\
				&&\hspace{4.5mm}+s^3\ga^4\iint_Q\left(\psi_t^2-4|x|^{\al+2}\right)^2\vp^3w^2\df x\df t +4(\al+2)s^3\ga^3\iint_Q|x|^{2\al+2}\vp^3w^2\df x\df t.
			\end{eqnarray*}
			These imply that
			\begin{eqnarray*}
				&&2(2-\alpha)s\ga \iint_Q|x|^{2\al}\vp |\nabla w|^2\df x\df t+12s\ga^2\iint_Q|x|^{2\al}(x\cdot \nabla w)^2\df x\df t\\
				&&+8s\ga\iint_Q|x|^{2\al-2}\vp (x\cdot \nabla w)^2\df x\df t +2s\ga \be\iint_Q\vp w_t^2\df x\df t \\
				&&+2s^3\ga^4\iint_Q\left(\psi_t^2-4|x|^{\al+2}\right)^2\vp^3w^2\df x\df t +8(\al+2)s^3\ga^3\iint_Q|x|^{2\al+2}\vp^3w^2\df x\df t\\
				&&\leq 4s\ga \iint_{\Sigma} |x|^{2\al}  \vp (x\cdot \nu)|\pt_\nu w|^2\df S\df t +\|e^{s\vp}F\|^2\\
				&&\leq 4s\ga \iint_{\Sigma_+} |x|^{2\al}  \vp (x\cdot \nu)|\pt_\nu w|^2\df S\df t +\|e^{s\vp}F\|^2.
			\end{eqnarray*}
			Hence, we have the estimate
			\begin{equation}\label{A-2}
				\begin{split}
					&\iint_Qs\ga \vp \left(|x|^\al \nabla w\cdot \nabla w+|w_t|^2+s^2\ga^2\vp^2|w|^2\right)\df x\df t\\
					&\leq C \iint_Qe^{2s\vp}F^2\df x\df t+C\iint_{\Sigma_+} s\ga \vp|\pt_\nu w|^2\df S\df t.
				\end{split}
			\end{equation}

			Now, we  estimate
			\begin{equation*}
				\iint_{\Sigma_+}  s\ga \vp  |\pt_\nu w|^2\df S\df t\leq C  \int_0^T\int_{\mcO(\Ga_+,3\de)}   s\ga \vp\left(|x|^\al \nabla u\cdot \nabla u+w_t^2+w^2\right)\df x\df t .
			\end{equation*}

			By Lemma \ref{L2.3}, there exists a $C^2$-vector field $V$ such that
			\begin{equation*}
				V(x)=\nu(x), \ x\in\pt\Om, \ \mbox{ and } |V(x)|\leq 1, \ x\in\Om.
			\end{equation*}
			Choose a function $\rho(x)\in C^2(\ol\Om)$ satisfying
			\begin{equation*}
				0\leq \rho\leq 1 \mbox{ on }\R^N,\quad \rho\equiv 1 \mbox{ on } \mcO(\Ga_+, \de)\cap \Om, \quad \rho=0 \mbox{ on }\Om\se \mcO(\Ga_+, 2\de).
			\end{equation*}
			Set
			\begin{equation*}
				g=V\rho .
			\end{equation*}
			Multiplying the first equality of \eqref{EQ-2} by  $g\cdot \nabla w$ and integrating over  $Q$, we obtain
			\begin{equation*}
				\begin{split}
					\iint_Q\left(\pt_t^2w-\Div(|x|^\al \nabla w)\right) (g\cdot \nabla w)\df x\df t=\iint_Q F(g\cdot \nabla w)\df x\df t.
				\end{split}
			\end{equation*}

			The proof follows a similar approach to the proof of Theorem \ref{T2.2}. Since
			$w(0)=w(T)=0$ and $w|_{\Sigma}=0$,  we derive
			\begin{eqnarray*}
				\iint_Q \pt_t^2 w(g\cdot \nabla w)\df x\df t
				&&=\iint_Q\pt_t\left(w_t(g\cdot \nabla w)\right)\df x\df t-\f{1}{2}\iint_Qg\cdot \nabla w_t^2\df x\df t \\
				&&=-\f{1}{2}\iint_Q \Div(g w_t^2)\df x\df t+\f{1}{2}\iint_Q (\Div g) w_t^2\df x\df t =\f{1}{2}\iint_Q (\Div g) w_t^2\df x\df t.
			\end{eqnarray*}
			Since
			\begin{eqnarray*}
				&&-\iint_Q\Div(|x|^\al \nabla w)(g\cdot \nabla w)\df x\df t\\
				&&=-\iint_Q\Div\left(|x|^\al \nabla w (g\cdot \nabla w)\right)\df x\df t+\iint_Q|x|^\al \nabla w\cdot \nabla (g\cdot \nabla w)\df x\df t\\
				&&=-\iint_\Sigma |x|^\al (\nabla w\cdot \nu)(g\cdot \nabla w)\df S\df t+\iint_Q |x|^\al \left((Dg\nabla w)\cdot \nabla w+\f{1}{2} g\cdot \nabla |\nabla w|^2\right)\df x\df t\\
				&&=-\iint_{\Sigma} |x|^\al \vp |\pt_\nu w|^2\df x\df t+ \iint_Q |x|^\al (Dg\nabla w)\cdot \nabla w\df x\df t\\
				&&\hspace{4.5mm}+\f{1}{2}\iint_Q\Div\left(|x|^\al g|\nabla w|^2\right)\df x\df t-\f{1}{2}\iint_Q \Div(|x|^\al g)|\nabla w|^2\df x\df t\\
				&&=-\f{1}{2}\iint_{\Sigma} |x|^\al \vp |\pt_\nu w|^2\df x\df t+ \iint_Q |x|^\al (Dg\nabla w)\cdot \nabla w\df x\df t-\f{1}{2}\iint_Q \Div(|x|^\al g)|\nabla w|^2\df x\df t.
			\end{eqnarray*}
			From the above, using the definition of $\rho$, we obtain
			\begin{equation}\label{10.03.7}
				\begin{split}
					\iint_{\Sigma_+} \vp|\pt_\nu w|^2\df x\df t
					&\leq C\iint_{\mcO(\Ga_+,2\de)\ts (0,T)} \vp\left(w_t^2+|x|^\al \nabla w\cdot \nabla w\right)\df x\df t+C\iint_Q\vp F^2\df x\df t\\
					&\leq C\iint_{\om\ts (0,T)}\vp\left(w_t^2+|x|^\al \nabla w\cdot \nabla w\right)\df x\df t+C\iint_Q\vp F^2\df x\df t.
				\end{split}
			\end{equation}
			Combining this with \eqref{A-2} and using classical arguments, we can transform back to the original variable $z$ and conclude the following result:
			\begin{equation}\label{4.5}
				\begin{split}
					&\iint_Qe^{2s\vp}s\ga \vp \left(|x|^\al \nabla z\cdot \nabla z+|z_t|^2+s^2\ga^2\vp^2|z|^2\right)\df x\df t\\
					&\leq C\iint_{\om\ts (0,T)}e^{2s\vp}s\ga \vp\left( z_t^2+|x|^\al \nabla z\cdot \nabla z +s^2\ga^2\vp^2|z|^2 \right)\df x\df t+C\iint_Qe^{2s\vp} F^2\df x\df t.
				\end{split}
			\end{equation}

			It is worth noting that the aforementioned conclusions are contingent upon the establishment of assumption \ref{*}. Subsequently, we shall delve into the discussion of acquiring equation \eqref{2.3} in the absence of assuming Condition \ref{*}.
			
			Let$z=z\kappa + z(1-\kappa)$, $\hat{z}=z(1-\kappa)$ , where
			$$
			\kappa \in C^\infty (\Omega), \quad \kappa \equiv 1 \mbox{ in } B\left(0,\frac{\varepsilon }{2}\right), \quad \kappa \equiv 0 \mbox{ in } \Om\backslash B(0,\varepsilon ), \quad 0\le \kappa \le 1.
			$$
			Then   $\hat{z} \subset \Om\backslash B(0,\varepsilon )$, and
			\begin{eqnarray*}
				\pt_t^2 \hat{z}-\Div\left(|x|^\alpha \nabla \hat{z}\right)
				&&=  \pt_t^2 z(1-\kappa)-\Div\left(|x|^\alpha \nabla\left[  z(1-\kappa)\right] \right)\\
				&&= \pt_t^2 z(1-\kappa)-\Div\left(|x|^\alpha \left[ \nabla z(1-\kappa) + z \nabla (1-\kappa) \right] \right)\\
				&&= \pt_t^2 z(1-\kappa)-\Div\left(|x|^\alpha \nabla z(1-\kappa)\right) -\Div\left(|x|^\alpha z \nabla (1-\kappa)  \right)\\
				&&= \pt_t^2 z(1-\kappa)-\Div\left(|x|^\alpha \nabla z \right)(1-\kappa) + |x|^\alpha \nabla z \cdot \nabla \kappa + \Div\left(|x|^\alpha z \nabla \kappa  \right)\\
				&&= F(1-\kappa) + |x|^\alpha \nabla z \cdot \nabla \kappa + \Div\left(|x|^\alpha z \nabla \kappa  \right)\\
				&&=:\hat{F}.
			\end{eqnarray*}
			Now, given that we have $\pt_t^2 \hat{z}-\Div\left(|x|^\alpha \nabla \hat{z}\right)=\hat{F}$ and $\hat{z}$
			satisfies \eqref{*}, we derive the following inequality:
			\begin{equation}\label{4.6}
				\begin{split}
					&\iint_Qe^{2s\vp}s\ga \vp \left(|x|^\al \nabla \hat{z}\cdot \nabla \hat{z}+|\hat{z}_t|^2+s^2\ga^2\vp^2|\hat{z}|^2\right)\df x\df t\\
					&\leq C\iint_{\om\ts (0,T)}e^{2s\vp}s\ga \vp\left( \hat{z}_t^2+|x|^\al \nabla \hat{z}\cdot \nabla \hat{z} +s^2\ga^2\vp^2|\hat{z}|^2 \right)\df x\df t+C\iint_Qe^{2s\vp} \hat{F}^2\df x\df t.
				\end{split}
			\end{equation}
			From this, we can deduce that
			\begin{eqnarray*}
				\iint_Qe^{2s\vp}s\ga \vp |x|^\al \nabla z\cdot \nabla z dx dt
				&=&\iint_Qe^{2s\vp}s\ga \vp |x|^\al \nabla \left( z\xi + \hat{z}\right) \cdot \nabla \left( z\xi + \hat{z}\right) dx dt\\
				&\le& \iint_Qe^{2s\vp}s\ga \vp |x|^\al \nabla \hat{z}\cdot \nabla \hat{z} dx dt + \iint_Qe^{2s\vp}s\ga \vp |x|^\al \nabla \left(z\xi \right) \cdot \nabla \left(z\xi \right) dx dt\\
				&+& 2 \iint_Qe^{2s\vp}s\ga \vp |x|^\al \nabla \left(z\xi \right) \cdot \nabla \hat{z} dx dt\\
				&\le& 2\iint_Qe^{2s\vp}s\ga \vp |x|^\al \nabla \hat{z}\cdot \nabla \hat{z} dx dt + 2\iint_Qe^{2s\vp}s\ga \vp |x|^\al \nabla \left(z\xi \right) \cdot \nabla \left(z\xi \right) dx dt.
			\end{eqnarray*}
			Similarly, we obtain
			\begin{equation*}
				\begin{split}
					\iint_Qe^{2s\vp}s\ga \vp |z_t|^2 dx dt
					&\le 2\iint_Qe^{2s\vp}s\ga \vp |(z \xi)_t | ^2 dx dt + 2\iint_Qe^{2s\vp}s\ga \vp |\hat{z}_t|^2  dx dt,
				\end{split}
			\end{equation*}
			and
			\begin{equation*}
				\begin{split}
					\iint_Qe^{2s\vp}s\ga \vp |z|^2 dx dt
					&\le 2\iint_Qe^{2s\vp}s\ga \vp |z\xi|^2 dx dt + 2\iint_Qe^{2s\vp}s\ga \vp |\hat{z}| ^2 dx dt,
				\end{split}
			\end{equation*}
			Utilizing \eqref{4.6} and the properties of  $\kappa$, we derive
			\begin{eqnarray*}
				&&\iint_Qe^{2s\vp}s\ga \vp \left(|x|^\al \nabla z\cdot \nabla z+|z_t|^2+s^2\ga^2\vp^2|z|^2\right)\df x\df t\\
				&&\le C\iint_Qe^{2s\vp}s\ga \vp \left(|x|^\al \nabla \hat{z}\cdot \nabla \hat{z}+|\hat{z}_t|^2+s^2\ga^2\vp^2|\hat{z}|^2\right)\df x\df t\\
				&&+
				C\iint_Qe^{2s\vp}s\ga \vp \left(|x|^\al \nabla (z \xi)\cdot \nabla (z \xi)+|(z \xi)_t|^2+s^2\ga^2\vp^2|z \xi|^2\right)\df x\df t\\
				&&\le
				C\iint_{\om\ts (0,T)}e^{2s\vp}s\ga \vp\left( \hat{z}_t^2+|x|^\al \nabla \hat{z}\cdot \nabla \hat{z} +s^2\ga^2\vp^2|\hat{z}|^2 \right)\df x\df t+C\iint_Qe^{2s\vp} \hat{F}^2\df x\df t\\
				&&+C\iint_{\om\ts (0,T)}e^{2s\vp}s\ga \vp\left( z_t^2+|x|^\al \nabla z\cdot \nabla z +s^2\ga^2\vp^2|z|^2 \right)\df x\df t\\
				&&\le C\iint_{\om\ts (0,T)}e^{2s\vp}s\ga \vp\left( z_t^2+|x|^\al \nabla z\cdot \nabla z +s^2\ga^2\vp^2|z|^2 \right)\df x\df t +C\iint_Qe^{2s\vp} \hat{F}^2\df x\df t.
			\end{eqnarray*}
			From this, we can deduce \eqref{2.3}. 		
		\end{proof}

		\section{Unique continuation and approximate controllability}\label{Se5}
		
		After establishing the Carleman estimate, we can readily draw conclusions regarding unique continuation and approximate controllability. Recall that our nonhomogeneous dual equation \eqref{EQ-2} is formulated as follows:
		\begin{equation*}
			\begin{cases}
				\pt_t^2z-\Div(|x|^\al \nabla z)=F, &\mbox{in } Q,\\
				z(0)=z_0, \pt_tz(0)=z_1, &\mbox{in }\pt Q,\\
				z=0, &\mbox{on }\Sigma.
			\end{cases}
		\end{equation*}
		Next, we present the following theorem, which is equivalent to establishing the unique continuation property given by
		\begin{equation}\label{10.03.1}
			E(0)\leq C\left( \|\pt_\nu z\|_{L^2(\Sigma)}+\|F\|_{L^2(Q)}^2\right)
		\end{equation}
		for some constant $C>0$ that  depends on the solution $z$ of \eqref{EQ-2}. Let
		\begin{equation*}
			T_0=\f{2}{\sqrt{\e^{\al+1}}}\max_{x\in\Om}|x|,\quad T\geq T_0.
		\end{equation*}
		We choose $\de>0$ and $0<\be<\f{\e^\al}{5}$ such that
		\begin{equation*}
			\e^{\al+1}T^2>4\max_{x\in \Om}|x|^2+4\de,
		\end{equation*}
		and
		\begin{equation*}
			\be T^2>4\max_{x\in\Om}|x|^2+4\de, \ 0<\be <\e^{\al+1}.
		\end{equation*}
		It is possible to select $T$ that satisfies the aforementioned conditions.
		
		\begin{theorem}\label{Carleman3}
			If $z$ is  a solution of \eqref{EQ-2} and  $\partial_{\nu} z=0$ on $\Sigma$, then $z=0$ in $Q$.
		\end{theorem}
		\begin{proof}			
			Assume $z$ is a solution of \eqref{EQ-2}. If $E(0)=0$, we have already proven the inequality \eqref{10.03.1}. Given the absolute continuity of
			$z_t\in L^2(Q), |x|^\al\nabla z\cdot \nabla z\in L^1(0,T; L^2(\Om))$, and $z\in L^2(Q)$, there exists $\varepsilon>0$
			such that $B(0,3\varepsilon)\s \Om$, and we have the following inequalities:
			\begin{equation*}
				\|z_t\|_{L^2(B(0,3\e)\ts (0,T))}\leq \f{1}{2}\|z_t\|_{L^2(Q)},\quad \||x|^\al \nabla z\cdot \nabla z\|_{L^1(0,T; B(0,3\e))}\leq \f{1}{2}\||x|^\al \nabla z\cdot \nabla z\|_{L^1(0,T; L^2(\Om))},
			\end{equation*}
			and
			\begin{equation*}
				\begin{split}
					&\|z_t(2^{-1}T)\|_{L^2(B(0,3\e)\ts (0,T))}\leq 2^{-1}\|z_t(2^{-1}T)\|_{L^2(Q)}, \\
					&\||x|^\al \nabla z(2^{-1}T)\cdot \nabla z(2^{-1}T)\|_{L^1(0,T; B(0,3\e))}\leq \f{1}{2}\||x|^\al \nabla z(2^{-1}T)\cdot \nabla z\|_{L^1(0,T; L^2(\Om))},
				\end{split}
			\end{equation*}
			as well as
			\begin{equation*}
				\|z\|_{L^2(B(0,3\e)\ts (0,T)}\leq \f{1}{2}\|z\|_{L^2(Q)}, \quad \|z(2^{-1}T)\|_{L^2(B(0,3\e)\ts (0,T)}\leq \f{1}{2}\|z(2^{-1}T)\|_{L^2(Q)}.
			\end{equation*}
			Define the function
			\begin{equation*}
				\psi(x,t)=|x|^2-\be\left(t-\f{T}{2}\right)^2.
			\end{equation*}
			It is evident that
			\begin{equation*}
				\psi(x,0)<-\de, \ \psi(x,T)<-\de \mbox{ for all } x\in \Om,
			\end{equation*}
			as well as
			\begin{equation*}
				\psi\left(x,\f{T}{2}\right)=|x|^2\geq 0, \ x\in \Om.
			\end{equation*}
			Hence, there exist constants $\eta>0$ and $ \wh \de>0$ such that
			\begin{alignat}{2}
				& \psi(x,t)\leq -2\wh \de, \mbox{ for all } x\in \Om,  t\in (0,2\eta)\cup (T-2\eta, T), \label{10.03.2}\\
				& \psi(x,t)\geq -\wh \de, \mbox{ for all }  x\in \Om, \left|t-\f{T}{2}\right|\leq \eta. \label{10.03.3}\\
				&
				\begin{cases}
					\|z_t(t)\|_{B(0,3\e)}\leq \f{1}{4}\|z_t(t)\|_{L^2(\Om)}, \\ \||x|^\al \nabla z(t)\cdot \nabla z(t)\|_{L^2(B(0,3\e)}\leq \f{1}{4}\||x|^\al \nabla z(t)\cdot \nabla z(t)\|, \\
					\|z(t)\|_{L^2(B(0,3\e))}\leq \f{1}{4}\|z(t)\|_{L^2(\Om)},
				\end{cases}\ \fa \left|t-\f{T}{2}\right|\leq \eta. \label{10.03.5}
			\end{alignat}
			Choose a cut-off function $\xi\in C^\iy(\R)$  such that
			$0\leq \xi\leq 1, \xi=1$ in $(2\eta, T-2\eta)$, and $\xi=0$ in $(0,\eta)\cup (T-\eta, T)$.   Additionally, select a cut-off function
			$\zeta\in C^\iy(\R^N)$ such that $0\leq \zeta\leq 1, \zeta=1$ in $\R^N-B(0,2\e)$, and $\zeta=0$ in $B(0,\e)$.

			Define the function
			\begin{equation*}
				w=\xi(t)\zeta(x) z(x,t), \ (x,t)\in Q.
			\end{equation*}
			It can be readily verified that
			\begin{equation*}
				\begin{cases}
					\hspace{-1.5mm}
					\begin{array}{lll}
						\pt_t^2w-\Div(|x|^\al \nabla w)=\xi\zeta F
						&+\hspace{1.5mm}
						2\xi_t(t)\zeta(x)z_t+\xi_{tt}(t)\zeta(x)z\\
						&+\hspace{1.5mm}2\xi(t)|x|^\al \nabla \zeta(x)\cdot \nabla z+\xi(t) \Div(|x|^\al \nabla \zeta(x))z,
					\end{array} & \mbox{in }Q,\\
					w(0)=w(T)=w_t(0)=w_t(T)=0, & \mbox{in }\Om,\\
					w=0, & \mbox{on }\Sigma.
				\end{cases}
			\end{equation*}
			Applying the Carleman estimate to the function $w$, we derive
			\begin{eqnarray*}
				&& \iint_Q  s\left(|x|^\al \nabla w\cdot \nabla w+s|w_t|^2+s^3|w|^2\right)\df x\df t\\
				&&\leq C\iint_Q \left|2\xi_t(t)\zeta(x)z_t+\xi_{tt}(t)\zeta(x)z\right|^2e^{2s\vp}\df x\df t+C\iint_Q|\xi\zeta F|^2 e^{2s\vp}\df x\df t\\
				&&\hspace{4.5mm}+C\iint_Q\left|2\xi(t)|x|^\al \nabla \zeta(x)\cdot \nabla z+\xi(t) \Div(|x|^\al \nabla \zeta(x))z\right|^2e^{2s\vp}\df x\df t+C\int_{\Sigma_+} s|\pt_\nu w|^2 \df S\df t
			\end{eqnarray*}
			for any $s\geq s_*$.  Consequently, for any  $s\geq s_*$, we have
			\begin{equation}\label{10.03.4}
				\begin{split}
					&\int_{\f{T}{2}-\eta}^{\f{T}{2}+\eta}\int_{\Om\se B(0,3\e)} s\left(|x|^\al \nabla z\cdot \nabla z +s|z_t|^2+s^3|z|^2\right)e^{2s\vp}\df x\df t\\
					&\leq C\iint_Q \left|2\xi_t(t)\zeta(x)z_t+\xi_{tt}(t)\zeta(x)z\right|^2e^{2s\vp}\df x\df t+C\iint_Q|\xi\zeta F|^2 e^{2s\vp}\df x\df t\\
					&\hspace{4.5mm}+C\iint_Q\left|2\xi(t)|x|^\al \nabla \zeta(x)\cdot \nabla z+\xi(t) \Div(|x|^\al \nabla \zeta(x))z\right|^2e^{2s\vp}\df x\df t+C\int_{\Sigma_+} s|\pt_\nu z|^2e^{2s\vp}\df S\df t.
				\end{split}
			\end{equation}
			Set
			\begin{equation*}
				A_0:=e^{-\ga\wh \de},\quad A_1:=e^{-2\ga \wh \de}.
			\end{equation*}
			Since $\xi_t,\xi_{tt}$ are supported in $(0,2\eta)\cup (T-2\eta, T)$, by \eqref{10.03.2}, we obtain
			\begin{equation*}
				\begin{split}
					&\iint_Q \left|2\xi_t(t)\zeta(x)z_t+\xi_{tt}(t)\zeta(x)z\right|^2e^{2s\vp}\df x\df t\leq \iint_Q \left(|z_t|^2+z^2\right) e^{2sA_1}\df x\df t\leq Ce^{2sA_1}\iint_Q\left(|z_t|^2+z^2\right)\df x\df t.
				\end{split}
			\end{equation*}
			Since $\nabla \zeta$ is supported in $B(0,2\e)\se B(0,\e)$, we obtain
			\begin{eqnarray*}
				&& \iint_Q\left|2\xi(t)|x|^\al \nabla \zeta(x)\cdot \nabla z+\xi(t) \Div(|x|^\al \nabla \zeta(x))z\right|^2e^{2s\vp}\df x\df t\\
				&&\leq  C \iint_{B(0,2\e)\ts (0,T)} \e^{-2}\left(|x|^\al \nabla z\cdot \nabla z+z^2\right) e^{2s\vp}\df x\df t.
			\end{eqnarray*}
			By the definition of  $\e$, if $s$ is sufficiently large, the relevant term can be absorbed by the left-hand side of \eqref{10.03.4}. Hence, we obtain
			\begin{eqnarray*}
				&&e^{2A_0s}\int_{\f{T}{2}-\eta}^{\f{T}{2}+\eta}\int_{\Om\se B(0,3\eta)}\left(|x|^\al \nabla z\cdot \nabla z+z_t^2\right)\df t\\
				&&\leq Ce^{Cs}\int_{\Sigma_+}|\pt_\nu z|^2\df S\df t+Ce^{Cs}\|F\|_{L^2(Q)}^2+Ce^{2A_1s}\iint_Q\left(|z_t|^2+z^2\right)\df x\df t.
			\end{eqnarray*}
			By Remark \ref{R1} and \eqref{10.03.5}, we have
			\begin{eqnarray*}
				e^{2A_0s}\int_{\f{T}{2}-\eta}^{\f{T}{2}+\eta} E(t)\df t
				&&\leq 4e^{2A_0s}\int_{\f{T}{2}-\eta}^{\f{T}{2}+\eta}\int_{\Om\se B(0,3\e)}\left(|x|^\al \nabla z\cdot \nabla z+z_t^2\right)\df t\\
				&&\leq Ce^{Cs}\|\pt_\nu z\|_{L^2(\Sigma_+)}^2+Ce^{Cs}\|F\|_{L^2(Q)}^2+Ce^{2A_1s}E(0)T.
			\end{eqnarray*}
			Again, by Remark \ref{R1}, we derive
			\begin{equation*}
				E(0)\leq C(E(t)+\|F\|_{L^2(Q)}^2),
			\end{equation*}
			which implies
			\begin{equation*}
				E(0)-C\|F\|_{L^2(Q)}^2\leq CE(t), \ t\in [0,T].
			\end{equation*}
			Therefore,
			\begin{equation*}
				2\eta E(0)e^{2A_0s}-2C\eta \|F\|_{L^2(Q)}^2e^{2A_0s}\leq Ce^{Cs}\|\pt_\nu z\|_{L^2(\Sigma_+)}^2+Ce^{Cs}\|F\|_{L^2(Q)}^2+Ce^{2A_1s}E(0)T.
			\end{equation*}
			This leads to
			\begin{equation*}
				E(0)=C\left(e^{Cs}\|\pt_\nu z\|_{L^2(\Sigma_+)}^2+e^{Cs}\|F\|_{L^2(Q)}^2+e^{-2(A_0-A_1)s}E(0)\right),
			\end{equation*}
			and hence,
			\begin{equation*}
				E(0)\leq C\left(\|\pt_\nu z\|_{L^2(\Sigma_+)}^2+\|F\|_{L^2(Q)}^2\right)
			\end{equation*}
			by choosing $s$ large enough  such that
			\begin{equation*}
				Ce^{-2(A_0-A_1)s}\leq \f{1}{2}.
			\end{equation*}
			This proves \eqref{10.03.1}.
		\end{proof}
		
		Similarly, proving Theorem \ref{Carleman5} is equivalent to demonstrating the following unique continuation property:
		\begin{equation*}
			E(0)\leq C\left(\iint_{\om\ts (0,T)} z^2\df x\df t+\|F\|_{L^2(Q)}^2\right).
		\end{equation*}
		for some constant $C>0$ that  depends on the solution $z$ of \eqref{EQ-2}.
		\begin{proof}[\rm \bf{{\bf Proof of Theorem \ref{Carleman5}.}}]	
			We start with \eqref{10.03.4}. By applying \eqref{10.03.7}, we obtain the following inequality:
			\begin{eqnarray}\label{10.03.6}
				&&\int_{\f{T}{2}-\eta}^{\f{T}{2}+\eta}\int_{\Om\se B(0,3\e)} s\left(|x|^\al \nabla z\cdot \nabla z +|z_t|^2+|z|^2\right)e^{2s\vp}\df x\df t \nonumber\\
				&&\leq C\iint_Q \left|2\xi_t(t)\zeta(x)z_t+\xi_{tt}(t)\zeta(x)z\right|^2e^{2s\vp}\df x\df t+C\iint_Q|\xi\zeta F|^2 e^{2s\vp}\df x\df t\nonumber \\
				&&\hspace{4.5mm}+C\iint_Q\left|2\xi(t)|x|^\al \nabla \zeta(x)\cdot \nabla z+\xi(t) \Div(|x|^\al \nabla \zeta(x))z\right|^2e^{2s\vp}\df x\df t+C\int_{\Sigma_+} s|\pt_\nu z|^2e^{2s\vp}\df S\df t
				\nonumber \\
				&&\leq Ce^{2A_1s}\iint_Q\left(z_t^2+z^2\right)\df x\df t+Ce^{Cs}\|F\|_{L^2(Q)}^2 \nonumber\\
				&&\hspace{4.5mm} +C\iint_{\om\ts (0,T)}s\left(|x|^\al \nabla z\cdot \nabla z+|z_t|^2+|z|^2\right)\df x\df t.
			\end{eqnarray}
			Next, we derive the following inequality:
			\begin{eqnarray*}
				&& e^{2sA_0}\int_{\f{T}{2}-\eta}^{\f{T}{2}+\eta}\int_\Om \left(|x|^\al \nabla z\cdot \nabla z +|z_t|^2+|z|^2\right)e^{2s\vp}\df x\df t\\
				&&\leq Ce^{2A_1s}\iint_Q\left(z_t^2+z^2\right)\df x\df t+Ce^{Cs}\|F\|_{L^2(Q)}^2\\
				&&\hspace{4.5mm} +C\iint_{\om\ts (0,T)}s\left(|x|^\al \nabla z\cdot \nabla z+|z_t|^2+|z|^2\right)\df x\df t.
			\end{eqnarray*}
			By using similar reasoning and manipulations, we can conclude that
			\begin{equation*}
				E(0)\leq C\left(\iint_{\om\ts (0,T)} z^2\df x\df t+\|F\|_{L^2(Q)}^2\right).
			\end{equation*}
			This completes the proof of the theorem.
		\end{proof}
		
		\section{Observability inequality and exact controllability}\label{Se6}

		In this section, we derive the observability inequality for the system described by \eqref{EQ-1} and establish its exact controllability. We define the problem \eqref{EQ-1} to be exactly controllable if, for any initial conditions
		$(u_0,u_1)\in L^2(\Omega) \times \mcH^{-1}(\Om)$, there exists a unique solution to \eqref{EQ-1} that starts from
		$(u_0,u_1)$  and satisfies the terminal conditions:
		\begin{equation*}
			u(\cdot,T)=0 \ \mbox{and} \ u_t (\cdot,T)=0 \ \mbox{in} \ \Omega.
		\end{equation*}

		First, we introduce the operator $J$ defined as
		\begin{equation*}
			J: \mcH=\mcH_0^1(\Om) \times L^2(\Omega) \to \mcH^{-1}(\Om) \times L^2(\Omega)=\mcH'
		\end{equation*}
		where $J$ maps $(z_0,z_1)$ to:
		\begin{equation*}
			J(z_0,z_1):=\left( u_t (\cdot,0), -u(\cdot,0)\right) .
		\end{equation*}
		Here, $\mcH'$  denotes the dual space of  $\mcH$, with the duality pairing between
		$\mcH$ and $\mcH'$,   denoted by $\left\langle  \cdot,\cdot \right\rangle_{\mcH,\mcH'}$. Note that we identify the dual of
		$L^2(\Omega) \times L^2(\Omega)$  with itself.

		\begin{proof}[\rm \bf{{\bf Proof of Theorem \ref{Carleman4}.}}]	We demonstrate the controllability of the system on the domain
			$\om$.
			
			Assume $0\in\om$. According to Theorem \ref{T2.2}, the following system:
			\begin{equation*}
				\begin{cases}
					\pt_t^2u-\Div(|x|^\al \nabla u)=\chi_\omega f, &\mbox{in } Q,\\
					u(\cdot,T)=0, \pt_tu(\cdot,T)=0, &\mbox{in }\pt Q,\\
					u=0, &\mbox{on }\Sigma,
				\end{cases}
			\end{equation*}
			admits a unique solution $u$ that satisfies:
			\begin{equation*}
				u(\cdot,0)\in L^2(\Omega) \ \mbox{and} \ u_t (\cdot,0)\in \mcH^{-1}(\Om),
			\end{equation*}			
			By setting $f=\chi_\om z$,  and multiplying $z$ to the first equation of \eqref{EQ-1}, we derive
			\begin{eqnarray*}
				&&\iint_Q\left(\pt_t^2u-\Div(|x|^\al \nabla u)\right)z\df x\df t=\iint_Q \chi_\om z^2\df x\df t\\
				=&& \int_{\Omega} \left( z u_t - z_t u \right) \df x \bigg|_0^T
				+\int_0^T \int_{\Omega} (\pt_t^2z -\Div(|x|^\al \nabla z) u \df x\df t + \iint_{\Sigma} u |x|^\al \nabla z \cdot \nu - z |x|^\al \nabla u \cdot \nu  \df S \df t\\
				=&& \int_{\Omega} \left( z_t(\cdot,0) u(\cdot,0) - z(\cdot,0) u_t(\cdot,0) \right) \df x =\iint_Q \chi_\om z^2\df x\df t.
			\end{eqnarray*}
			Therefore,
			\begin{equation*}
				\begin{split}
					\left\langle J(z_0,z_1), (z_0,z_1)\right\rangle _{\mcH,\mcH'}= \iint_Q \chi_\om z^2\df x\df t.
				\end{split}
			\end{equation*}
			By Theorem \ref{T2.2}, we have
			\begin{equation*}
				\begin{split}
					\left\langle J(z_0,z_1), (z_0,z_1)\right\rangle _{\mcH,\mcH'}\ge C \|(z_0,z_1)\|^2_{\mcH} .
				\end{split}
			\end{equation*}
			Applying the Lax-Milgram theorem to the linear operator $J$, and following the standard proofs in \cite{zuazua}, we conclude the exact controllability of equation \eqref{EQ-1}.
			%
		\end{proof}
		
		%
		%
		%
		%

\section{Concluding Remarks}\label{Se7}

 In this study, we establish the controllability for a class of high-dimensional hyperbolic equations characterized by a single interior degenerate point. This is achieved by employing the Carleman estimate method to derive the corresponding observability inequality. It is important to note that the selection of the multiplier is highly contingent upon the position of the degenerate point. In scenarios involving two or more interior degenerate points, the methodology for constructing an appropriate multiplier remains ambiguous. A plausible approach is to isolate the degenerate interior points, thereby enabling the energy of the degenerate hyperbolic equation to be controlled through both the boundary control domain and the isolated domains surrounding the interior degeneracies. Nevertheless, the rigorous implementation of this strategy remains an open question.

Future research may investigate the case where the weight function $w$ features countably many interior degenerate points $\{x_n\}_{n \in \mathbb{N}}$. In such cases, we select a control domain $\omega$ such that
\[
\omega \supset \bigcup_{n \in \mathbb{N}} B(x_n, r_n),
\]
where $B(x_n, r_n) \cap B(x_m, r_m) = \emptyset$ for all distinct $m, n \in \mathbb{N}$. Alternatively, the set $\{x_n\}_{n\in\mathbb{N}}$ could represent all rational points on $\mathbb{S}^{N-1}\setminus\Omega$. Despite these considerations, the challenge of constructing an appropriate multiplier persists as an unresolved issue.

	\end{document}